\documentclass[a4paper,11pt]{amsart}

\usepackage{geometry}
\usepackage{amssymb,mathrsfs,amsmath,amsthm}
\usepackage{enumitem}
\usepackage{autobreak}
\allowdisplaybreaks
\usepackage{bm}
\usepackage{bbm}

\usepackage{graphicx}
\usepackage{color}
\usepackage{booktabs}
\usepackage{hyperref}
\usepackage{cleveref}

\newtheorem{theorem}{Theorem}[section]
\newtheorem{lemma}{Lemma}[section]
\newtheorem{proposition}{Proposition}[section]
\newtheorem{corollary}{Corollary}[section]
\newtheorem{definition}{Definition}[section]
\newtheorem{remark}{Remark}[section]

\theoremstyle{definition}
\newtheorem{example}{Example}[section]

\numberwithin{equation}{section}
\numberwithin{figure}{section}

\makeatletter
\@namedef{subjclassname@2020}{\textup{2020} Mathematics Subject Classification}
\makeatother

\begin{document}

\title{Existence, Equivalence, and Convergence of Unbalanced Optimal Transport}

\author{Qing Zhao}
\address{School of Mathematics and Shing-Tung Yau Center, Southeast University, Nanjing 211189, P. R. China}
\email{zhao-qing@seu.edu.cn}

\author{Feida Jiang$^*$}
\address{School of Mathematics and Shing-Tung Yau Center, Southeast University, Nanjing 211189, P. R. China; Shanghai Institute for Mathematics and Interdisciplinary Sciences, Shanghai 200433, P. R. China}
\email{jiangfeida@seu.edu.cn}


\date{\today}
\thanks{$^*$ Corresponding author: Feida Jiang.}

%\thanks{{\bf {Maybe we can submit this paper to IMRN or other good journals.}}}


\begin{abstract}
We study an unbalanced optimal transport problem with a spatially independent signed source.
The target density is continuous on a bounded domain
or continuous and integrable on $\mathbb R^n$. We establish sequential weak lower semicontinuity
at the critical Sobolev regularity $W^{1,n}$
and existence of minimizers under prescribed
distortion and Jacobian mass bounds, provided
a finite-energy admissible pair exists.
For domain-preserving, uniformly bi-Lipschitz affine
flows, we prove an energy-preserving static-dynamic
correspondence and stability of the reconstruction.
Within this class, sequential $\Gamma$-convergence
identifies the strong source-penalty limit as transport
minimization among pairs of minimum source cost.
This yields optimal-cost asymptotics and subsequential
convergence of minimizers and their reconstructed paths.
When a constant source is admissible, optimal sources
converge strongly in $L^2$ to the constant determined
by the endpoint mass difference.
We establish local constant-source feasibility
for H\"older endpoint densities near a fixed reference
pair and give a quantitative obstruction showing
that constant-source feasibility may fail even when admissible pairs exist.
\end{abstract}

\maketitle

\section{Introduction}\label{sec1}
Optimal transport provides a variational framework
for comparing and redistributing mass.
In its classical formulation, the source and target
measures have the same total mass.
Unbalanced optimal transport relaxes this requirement
by allowing mass creation and removal.

Let $\Omega\subset\mathbb R^n$ be a bounded domain,
let $\rho_0$ and $\rho_1$ denote the nonnegative
initial and target densities, and let $\alpha>0$.
For a map $M\in W^{1,n}(\Omega;\mathbb R^n)$
and a spatially independent source
$f\in L^2(0,1)$, the static energy is formally given by
\begin{equation}\label{e3}
\begin{aligned}
E_\alpha(M,f)
={}&
\int_\Omega |M(x)-x|^2\rho_0(x)\,dx
+\frac1\alpha\int_0^1|f(t)|^2\,dt\\
&+
\int_0^1\int_0^t f(s)
\int_\Omega |M(x)-x|^2
\det\bigl(s\nabla M(x)+(1-s)I\bigr)
\,dx\,ds\,dt,
\end{aligned}
\end{equation}
where $I$ is the identity matrix.
The first term is the displacement cost of the
initial mass, the second penalizes the source,
and the last accounts for its contribution along
the affine interpolation from the identity to $M$. The map and the source satisfy the endpoint constraint
\begin{equation}\label{eq1.2}
\rho_1(M(x))\det\nabla M(x)
=
\rho_0(x)+\int_0^1 f(t)J_t[M](x)\,dt.
\end{equation}
Its distributional interpretation and the remaining
admissibility conditions are specified below. 

Write $
X_t[M]:=(1-t)\mathrm{Id}+tM$, $J_t[M]:=\det\bigl((1-t)I+t\nabla M\bigr)$
and define the intermediate state
$$
R_t[M,f](x)
:=
\rho_0(x)+\int_0^t f(s)J_s[M](x)\,ds.
$$
For pairs with nonnegative intermediate states, define
$$
G(M,f):=
\int_0^1\int_\Omega
|M(x)-x|^2R_t[M,f](x)\,dx\,dt
\in[0,+\infty]
$$
and
$$
E_\alpha(M,f):=
G(M,f)+\frac1\alpha\|f\|_{L^2(0,1)}^2.
$$
This definition agrees with \eqref{e3}
for bounded-range maps and whenever the terms
in that expression are absolutely integrable.
We study existence for this static problem, its correspondence with dynamic transport under prescribed affine-flow conditions, and the limits
selected by strong source penalization.
We also give sufficient conditions and explicit
examples verifying the required feasibility assumptions. We first review the related literature and then
describe our approach and main results.

\subsection{Background and related work}
The classical transport problem originates with
Monge~\cite{G}.
Kantorovich~\cite{L} introduced its relaxation
to transport plans, while Benamou and
Brenier~\cite{Y1} formulated quadratic optimal
transport through a continuity equation
and a kinetic action.
Existence of optimal maps has been studied
under various assumptions on the cost
and the endpoint measures
\cite{E3,W2,L11,L13,S12}.
For background on classical optimal transport, we refer the reader to \cite{C,C1}.

Several extensions of optimal transport have been
developed for finite nonnegative measures with
unequal total masses. One approach to unbalanced optimal transport extends
the Benamou-Brenier formulation by introducing
a source into the continuity equation and penalizing
mass creation and removal.
Chizat et al.~\cite{L1} introduced the
Wasserstein-Fisher-Rao metric, which combines
the quadratic Wasserstein action for spatial transport
with the Fisher-Rao action for mass variation.
In \cite{CPSV2018}, they studied a broader class
of dynamic transport costs defined by convex,
positively one-homogeneous action densities
and established equivalent static Kantorovich formulations in terms of semi-couplings. Static approaches include Kantorovich-type problems
over transport plans and Monge-type problems involving transport maps and mass scaling factors. Guittet~\cite{K}
studied extended Kantorovich norms with a fixed cost per unit of mass created or removed. Piccoli and Rossi~\cite{B} introduced generalized Wasserstein distances combining $W_p$ transport
with total variation penalties for mass variation, and subsequently established their dynamic characterization and duality properties in \cite{B1}.
Liero, Mielke, and Savar\'e~\cite{LMS2018}
developed the entropy-transport framework as a Kantorovich-type problem with entropy penalties on the marginals.
Savar\'e and Sodini~\cite{SS2024}
obtained formulations over homogeneous couplings
through convex relaxation and investigated
the corresponding Monge formulations.
For the Hellinger-Kantorovich distance,
Liero, Mielke, and Savar\'e~\cite{LMS2023}
established sufficient conditions for the solvability of the Monge problem and constructed transport-growth flows describing geodesics.
Gallou\"et, Ghezzi, and Vialard~\cite{GGV2025} studied regularity in the static entropy-transport framework and its associated Monge formulation, and proved static-dynamic equivalence for the Wasserstein-Fisher-Rao metric on closed Riemannian manifolds.

The present problem is motivated by
Gangbo et al.~\cite{bib1}.
They introduced a spatially independent
additive source with a quadratic penalty
into the dynamic transport problem
and derived an associated Monge formulation
through Lagrangian coordinates.
This leads to the energy \eqref{e3}
and the endpoint relation \eqref{eq1.2}.
A generalization allowing spatially dependent
sources was studied in \cite{W1}. Gangbo et al.~\cite{bib1} approach the problem
through its dynamic formulation and derive
an associated Monge formulation in Lagrangian
coordinates. We take the static formulation as our starting
point and establish existence of minimizers
under prescribed geometric constraints.
For domain-preserving, uniformly bi-Lipschitz
affine flows, we then prove equivalence with
the corresponding restriction of their dynamic
problem, including equality of minimum values
and a correspondence between minimizers.

\subsection{Outline of the approach}
The direct analysis of the static problem requires
compactness of minimizing sequences and stability
of the nonlinear constraints under weak convergence.
The arguments differ according to whether the
maps have a prescribed bounded range.

For bounded targets, the main difficulty is to
preserve the nonlinear endpoint constraint and
establish lower semicontinuity under weak convergence
of both maps and sources.
At the critical regularity $W^{1,n}$, distributional
continuity of Jacobians does not directly imply
convergence of the weighted energy.
We combine composition identities with weak continuity
of minors to obtain convergence of localized weighted
Jacobians.
Uniform approximation of the continuous target density
also allows passage to the limit in the endpoint
constraint.
The polynomial dependence of $J_t[M]$ on time
provides the uniform convergence needed to handle
its coupling with weakly convergent sources.
Nonnegativity of the intermediate states then permits
a cutoff argument yielding global lower semicontinuity, even when the sources change sign.
The distortion and Jacobian mass bounds control the gradients, while the bounded range controls
the maps themselves.
These estimates provide the compactness required
by the direct method.

For whole-space targets, the absence of a prescribed
bounded range creates two additional difficulties. The weights depending on the maps are no longer
uniformly bounded, and gradient bounds alone
do not control translations.
To address the first issue, local H\"older estimates
under the distortion bound yield locally uniform
convergence along weakly convergent sequences.
This preserves the weighted endpoint constraint
and allows the preceding energy argument to be
applied locally.
To obtain a full Sobolev bound, we use positive
initial mass to find a set of uniformly positive
measure on which the intermediate states remain
bounded below for a short time.
The transport energy then controls the displacement
on this set.
Combined with Poincar\'e's inequality and the gradient
bound, this controls the mean of the map and
establishes coercivity.

For the static-dynamic correspondence, we restrict
attention to domain-preserving, uniformly bi-Lipschitz
affine flows.
The change-of-variables formula reconstructs
dynamic densities and velocities from static pairs
and preserves the energy.
Conversely, testing the continuity equation with
functions transported by the inverse flow recovers
the static endpoint constraint.
Uniform control of the inverse flows, together
with weak continuity of minors and the polynomial
structure of the interpolation Jacobians, also
establishes stability of the reconstruction. The source-penalty limit is studied within this
same geometric class.
After subtracting the minimum source cost divided
by $\alpha$, compactness and lower semicontinuity
give the $\Gamma$-liminf inequality, while constant
sequences provide recovery sequences.
Comparison with admissible pairs of minimum source
cost identifies the limiting transport problem
and yields convergence of minimum values and
subsequential convergence of minimizers.
Convergence of the source norms upgrades their weak
convergence to strong convergence in $L^2(0,1)$.
The reconstruction stability then transfers these
limits to the corresponding dynamic paths.
When a constant source is admissible, the limiting
source is the constant determined by the endpoint
mass difference.

Finally, we investigate the existence of admissible pairs with constant source.
The Jacobian bounds give a necessary condition for constant-source admissibility.
For densities close to a fixed reference pair,
a solution operator for the divergence equation
and Schauder's fixed point theorem solve the
nonlinear endpoint constraint.
This construction produces constant-source pairs
satisfying the prescribed distortion and flow bounds,
so the existence, equivalence, and convergence
results apply without a separate nonemptiness
assumption in this local setting.

\subsection{Main results}
Throughout this paper, let $n\ge2$,
let $\Omega\subset\mathbb R^n$ be a bounded
open set with Lipschitz boundary, and let
$0\le\rho_0\in L^1(\Omega)$. We equip $
Y:=
\bigl(
W^{1,n}(\Omega;\mathbb R^n)
\cap L^\infty(\Omega;\mathbb R^n)
\bigr)\times L^2(0,1)$ with the topology $\tau$.
Here $(M_j,f_j)\xrightarrow{\tau}(M,f)$ means that $M_j\rightharpoonup M$ in $W^{1,n}(\Omega;\mathbb R^n)$, $M_j\rightharpoonup^*M$ in $L^\infty(\Omega;\mathbb R^n)$, and $f_j\rightharpoonup f$ in $L^2(0,1)$. We equip $
Y_1:=
W^{1,n}(\Omega;\mathbb R^n)\times L^2(0,1)$ with the product weak topology.
We use the following four sets of assumptions
and specify the corresponding admissible classes.

\begin{enumerate}
\renewcommand{\theenumi}{H\arabic{enumi}}
\renewcommand{\labelenumi}{(\theenumi)}

\item\label{hyp:bounded-data}
\textit{Bounded targets.}
Assume that $\Omega$ is convex and
$0\le\rho_1\in C(\overline\Omega)$.
Set
$$
X(\Omega;\overline\Omega)
:=
\left\{
M\in W^{1,n}(\Omega;\mathbb R^n)
\cap L^\infty(\Omega;\mathbb R^n):
M(x)\in\overline\Omega\ \text{a.e.}
\right\}.
$$
Let $\mathcal V^{\rho_1}$ consist of pairs $
(M,f)\in
X(\Omega;\overline\Omega)\times L^2(0,1)$
satisfying $
J_t[M]\ge0$ and $R_t[M,f]\ge0$ a.e. in $\Omega$ and for every fixed $t\in[0,1]$, and $
\rho_1(M)\det\nabla M=R_1[M,f]$ in $\mathcal D'(\Omega)$.
For $K>0$, define
$$
\mathcal V_K^{\rho_1}
:=
\left\{
(M,f)\in\mathcal V^{\rho_1}:
\begin{array}{l}
|\nabla M|^n\le K\det\nabla M\text{ a.e. in }\Omega,\\[1mm]
\displaystyle
\int_\Omega\det\nabla M\,dx\le|\Omega|
\end{array}
\right\},
$$
\item\label{hyp:whole-space-data}
\textit{Whole-space targets.}
Assume that $\Omega$ is connected,
$0\le\rho_1\in C(\mathbb R^n)\cap L^1(\mathbb R^n)$,
and $\int_\Omega\rho_0\,dx>0$.
Let $\widetilde{\mathcal V}^{\rho_1}$
satisfy the same source, nonnegativity,
and endpoint constraints as $\mathcal V^{\rho_1}$,
with the map requirement replaced by
$M\in W^{1,n}(\Omega;\mathbb R^n)$.
For $K,\Lambda>0$, define
$$
\mathcal V_{K,\Lambda}^{\rho_1}
:=
\left\{
(M,f)\in\widetilde{\mathcal V}^{\rho_1}:
\begin{array}{l}
|\nabla M|^n\le K\det\nabla M
\text{ a.e. in }\Omega,\\[1mm]
\displaystyle
\int_\Omega\det\nabla M\,dx\le\Lambda
\end{array}
\right\}.
$$
The distributional endpoint constraint includes
the requirement $
\rho_1(M)\det\nabla M\in L^1_{\mathrm{loc}}(\Omega)$.
\item\label{dy-am}\textit{Static-dynamic setting.}
Assume that $\Omega$ is convex and
$0\le\rho_1\in C(\overline\Omega)$. For $L\ge1$, let $\mathcal A_L^{\rho_1}$
consist of pairs $(M,f)\in\mathcal V^{\rho_1}$
with $M\in C(\overline\Omega;\overline\Omega)$
satisfying
\begin{equation}
\label{ex1.3}
\begin{gathered}
X_t[M](\overline\Omega)=\overline\Omega,\\
L^{-1}|x-y|
\le |X_t[M](x)-X_t[M](y)|
\le L|x-y|
\end{gathered}
\end{equation}
for every $x,y\in\overline\Omega$
and $t\in[0,1]$.

\item\label{hyp:local-holder}
\textit{Local H\"older setting.}
Fix $K>n^{n/2}$, $L>1$, and $0<\beta<\gamma<1$ and assume that $\Omega$ is a bounded convex domain with $C^{3,\beta}$ boundary. Let $\overline\rho\in C^{1,\gamma}(\overline\Omega)$
and $a\in\mathbb R$ satisfy $
\overline\rho\ge0$, $\overline\rho-a\ge0$, and $
\min_{\overline\Omega}
\left(\overline\rho-\frac a2\right)>0$.
The prescribed densities satisfy $
0\le\rho_0,\rho_1\in C^{0,\beta}(\overline\Omega)$ and $\int_\Omega(\rho_1-\rho_0)\,dx=a|\Omega|$. Write $d:=
\|\rho_0-\overline\rho+a\|_{C^{0,\beta}(\overline\Omega)}
+\|\rho_1-\overline\rho\|_{C^{0,\beta}(\overline\Omega)}$.
\end{enumerate}

We first establish lower semicontinuity and
existence in each of the settings
\textup{(\ref{hyp:bounded-data})} and
\textup{(\ref{hyp:whole-space-data})}.

\begin{theorem}[Lower semicontinuity]
\label{TH5.1}
Under \textup{(\ref{hyp:bounded-data})},
$E_\alpha$ is sequentially lower semicontinuous
on $\mathcal V^{\rho_1}$ with respect to $\tau$
for every $\alpha>0$.
\end{theorem}

\begin{theorem}[Existence for bounded targets]
\label{GTH5.2}
Under \textup{(\ref{hyp:bounded-data})},
let $K>0$ and assume that
$\mathcal V_K^{\rho_1}\ne\varnothing$.
Then $E_\alpha$ attains its minimum
on $\mathcal V_K^{\rho_1}$
for every $\alpha>0$.
\end{theorem}

\begin{theorem}[Whole-space targets]
\label{T5.3}
Under \textup{(\ref{hyp:whole-space-data})},
let $K,\Lambda>0$.
For every $\alpha>0$, $E_\alpha$
is sequentially weakly lower semicontinuous
on $\mathcal V_{K,\Lambda}^{\rho_1}$.
If this class contains a pair $(M_*,f_*)$
with $G(M_*,f_*)<\infty$,
then $E_\alpha$ attains its minimum
on $\mathcal V_{K,\Lambda}^{\rho_1}$.
\end{theorem}

We next state the dynamic correspondence
in the bounded-target setting.
Let $\mathfrak D_L^{\rho_0,\rho_1}$ denote
the dynamic class defined precisely
in Section~\ref{sec5.1}. Its triples satisfy $
0\le\rho\in C([0,1];L^1(\Omega))$, $
f\in L^2(0,1)$, the continuity equation $
\partial_t\rho+\operatorname{div}(\rho v)=f(t)$, the prescribed endpoint conditions, and the no-flux boundary condition. The velocity has a jointly Lipschitz representative
whose flow satisfies $
X_t=(1-t)\mathrm{Id}+tX_1$ and \eqref{ex1.3}
with $M=X_1$.
The dynamic energy is
$$
\mathscr E_\alpha(\rho,v,f)
:=
\int_0^1\int_\Omega
\rho(t,y)|v(t,y)|^2\,dy\,dt
+\frac1\alpha\|f\|_{L^2(0,1)}^2.
$$

\begin{theorem}[Static-dynamic equivalence]
\label{TH1.4}
Under \textup{(\ref{dy-am})},
let $L\ge1$ and assume that
$\mathcal A_L^{\rho_1}\ne\varnothing$.
Every $(M,f)\in\mathcal A_L^{\rho_1}$
determines a triple
$(\rho,v,f)\in\mathfrak D_L^{\rho_0,\rho_1}$
through
\begin{equation}\label{eq:dynamic-reconstruction}
\rho(t,X_t[M](x))
=
\frac{R_t[M,f](x)}{J_t[M](x)},
\qquad
v(t,X_t[M](x))=M(x)-x.
\end{equation}
Conversely, every triple in
$\mathfrak D_L^{\rho_0,\rho_1}$
arises from these formulas with $M=X_1$,
where $X_t$ is its flow.
The correspondence preserves the energy: $
\mathscr E_\alpha(\rho,v,f)=E_\alpha(M,f)$. For every $\alpha>0$, both minima are attained and
\begin{equation}\label{eq:static-dynamic-equality}
\min_{(M,f)\in\mathcal A_L^{\rho_1}}E_\alpha(M,f)
=
\min_{(\rho,v,f)\in\mathfrak D_L^{\rho_0,\rho_1}}
\mathscr E_\alpha(\rho,v,f).
\end{equation}
Static and dynamic minimizers correspond
under this construction.
\end{theorem}
We next describe the limit under strong source
penalization within the affine-flow class.
For $\mathcal A_L^{\rho_1}\ne\varnothing$, set
\[
b_L:=
\inf_{(M,f)\in\mathcal A_L^{\rho_1}}
\|f\|_{L^2(0,1)}^2,
\qquad
\mathcal S_L^{\rho_1}
:=
\left\{
(M,f)\in\mathcal A_L^{\rho_1}:
\|f\|_{L^2(0,1)}^2=b_L
\right\}.
\]
Define the extended functionals on $Y$ by
$$
\widehat{\mathcal F}_\alpha(M,f):=
\begin{cases}
E_\alpha(M,f)-b_L/\alpha,
& (M,f)\in\mathcal A_L^{\rho_1},\\
+\infty,&\text{otherwise},
\end{cases}
$$
and
$$
\widehat{\mathcal F}_0(M,f):=
\begin{cases}
G(M,f),&(M,f)\in\mathcal S_L^{\rho_1},\\
+\infty,&\text{otherwise}.
\end{cases}
$$

\begin{theorem}[Minimum-source-cost limit]\label{thm:minimum-source-limit}
Under \textup{(\ref{dy-am})},
let $L\ge1$ and assume that
$\mathcal A_L^{\rho_1}\ne\varnothing$.
The infimum defining $b_L$ is attained, and $G$ attains its minimum on $\mathcal S_L^{\rho_1}$.
Write $g_L:=\min_{(M,f)\in\mathcal S_L^{\rho_1}}G(M,f)$. Then
$\widehat{\mathcal F}_\alpha
\xrightarrow{\Gamma}\widehat{\mathcal F}_0$ as $\alpha\to0^+$ in the sequential sense with respect to $\tau$, and
\begin{equation}\label{eq:minimum-source-expansion}
\min_{\mathcal A_L^{\rho_1}}E_\alpha
=\frac{b_L}{\alpha}+g_L+o(1).
\end{equation}
For every sequence $\alpha_j\to0^+$ and corresponding
minimizers $(M_{\alpha_j},f_{\alpha_j})$ of $E_{\alpha_j}$,
there exist a subsequence, not relabeled, and a pair
$(M,f)\in\mathcal S_L^{\rho_1}$ with $G(M,f)=g_L$ such that
$$
M_{\alpha_j}\rightharpoonup M
\quad\text{in }W^{1,n}(\Omega;\mathbb R^n),
\qquad
M_{\alpha_j}\to M
\quad\text{uniformly on }\overline\Omega,
$$
and
$$
f_{\alpha_j}\to f
\quad\text{strongly in }L^2(0,1).
$$
Moreover, for any choice of minimizers,
\begin{equation}\label{eq:minimum-source-energy-limits}
G(M_\alpha,f_\alpha)\to g_L,
\qquad
\frac{\|f_\alpha\|_{L^2(0,1)}^2-b_L}{\alpha}\to0.
\end{equation}
\end{theorem}
The following theorem combines the preceding results
with the local feasibility construction
in Section~\ref{subsec:local-feasibility}.
It gives a unified existence, static-dynamic
equivalence, and convergence result for H\"older
densities near a fixed reference pair,
without a separate nonemptiness assumption.

\begin{theorem}[Existence, equivalence, and convergence
for locally feasible densities]
\label{cor:local-existence-convergence}
Under \textup{(\ref{hyp:local-holder})},
there exists $\varepsilon>0$, depending only on
$\Omega,n,\beta,\gamma,\overline\rho,a,K,L$,
such that the following conclusions hold
whenever $d<\varepsilon$. There exists a constant-source pair $
(M_*,a)\in
\mathcal V_K^{\rho_1}\cap\mathcal A_L^{\rho_1}$. For every $\alpha>0$, $E_\alpha$ attains
its minimum on both $\mathcal V_K^{\rho_1}$
and $\mathcal A_L^{\rho_1}$.
Minimizers on $\mathcal A_L^{\rho_1}$
correspond to minimizers on
$\mathfrak D_L^{\rho_0,\rho_1}$
through the energy-preserving reconstruction. The conclusions of
Theorem~\ref{thm:minimum-source-limit}
apply with $
b_L=a^2$ and $g_L=m_a:=
\min\left\{
G(M,a):(M,a)\in\mathcal A_L^{\rho_1}
\right\}$. In particular,
$$
\min_{\mathcal A_L^{\rho_1}}E_\alpha
=
\min_{\mathfrak D_L^{\rho_0,\rho_1}}
\mathscr E_\alpha
=
\frac{a^2}{\alpha}+m_a+o(1)
\qquad\text{as }\alpha\to0^+.
$$
For any choice of minimizers
$(M_\alpha,f_\alpha)$ on $\mathcal A_L^{\rho_1}$,
$\|f_\alpha-a\|_{L^2(0,1)}
=o(\sqrt{\alpha})$.
\end{theorem}

\subsection{Organization of the paper}
The paper is organized as follows.

Section~\ref{sec2} recalls the weak continuity and
variational convergence results used in the proofs.
Sections~\ref{sec3} and~\ref{sec4} establish lower semicontinuity
and existence for bounded and whole-space targets,
respectively.
Section~\ref{sec5} proves the static-dynamic correspondence,
its stability, and the source-penalty limits.
Section~\ref{sec6} gives a local feasibility result
for H\"older densities and explicit examples
illustrating nonemptiness and the role of constant-source admissibility.

\section{Notation and preliminaries}\label{sec2}
\subsection{Notation}
Throughout the paper, $\Omega\subset\mathbb R^n$
denotes the bounded open domain, with the geometric
assumptions specified in each setting, and $|\Omega|$ denotes its $n$-dimensional Lebesgue measure. $|A|$ denotes the Frobenius norm
of a matrix $A$, while $|x|$ denotes the Euclidean norm
of a vector $x$. Let $C_c(\Omega)$ denote the space of continuous real-valued functions on $\Omega$ with compact support. We denote by $\mathcal{D}(\Omega)$ the space of smooth functions with compact support in $\Omega$. $\operatorname{supp}\phi$ denotes  the support of $\phi$. Let $U$ and $V$ be open subsets of $\mathbb{R}^{n}$.  We write $V\Subset U$, if $V\subset\overline{{V}}\subset U$ and $\overline{V}$ is compact, and say $V$ is compactly contained in $U$. $\chi_{U}$ denotes the characteristic function of $U$, namely,
$$
\chi_{U}(x)
=
\begin{cases}
1, & x\in U,\\
0, & x\in\Omega\setminus U.
\end{cases}
$$

\subsection{Weak continuity of determinants and minors}
Since the formulation of the unbalanced optimal transport problem involves an integral containing the determinant, the convergence of determinants under nonsmooth mappings is essential for establishing lower semicontinuity. We therefore recall the notion of weak continuity of the determinant.
Let $u=(u^{1},\cdots,u^{n})\in C^{2}(\Omega;\mathbb{R}^{n})$. We denote
$$
\det \nabla u
= \frac{\partial (u^1,\cdots,u^n)}
{\partial (x_1,\cdots,x_n)}
$$
by the classical determinant of the Jacobian matrix. It can also be written as
$$
\det \nabla u
=
\sum_{\alpha=1}^{n}(-1)^{\alpha+1}
\frac{\partial u^1}{\partial x_\alpha}
\frac{\partial (u^2,\cdots,u^n)}
{\partial (x_1,\cdots,x_{\alpha-1},x_{\alpha+1},\cdots,x_n)}.
$$
Here ${\partial (u^2,\cdots,u^n)}/
{\partial (x_1,\cdots,x_{\alpha-1},x_{\alpha+1},\cdots,x_n)}
$ denotes the determinant of the corresponding minor matrix. For $2\leq s\leq n$, we define the generalized adjugate operator
$\operatorname{adj}_s\nabla u$ by
$$
\operatorname{adj}_s \nabla u = \left( (\operatorname{adj}_s \nabla u)^i_\alpha \right)_{1 \leq \alpha \leq \binom{n}{s}}^{1 \leq i \leq \binom{n}{s}},
$$
where
$$
(\operatorname{adj}_s\nabla u)^i_{\ \alpha}
=
(-1)^{i+\alpha}
\frac{\partial
(u^{i_1},\dots,u^{i_s})}
{\partial
(x_{\alpha_1},\dots,x_{\alpha_s})}.
$$
Here $(i_1,\dots,i_s)$ and $(\alpha_1,\dots,\alpha_s)$ are the $s$-tuples corresponding to the indices $i$ and $\alpha$, respectively.

\begin{definition}\cite{B3}
Let $\Omega\subset\mathbb{R}^n$ be an open set and let
$f_v,f\in L^1_{\mathrm{loc}}(\Omega)$.
We say that $f_v$ converges to $f$ in the sense of distributions, and write
$$
f_v\rightharpoonup f
\quad\text{in }\mathcal{D}'(\Omega),
$$
if
$$
\int_\Omega f_v(x)\varphi(x)\,dx
\longrightarrow
\int_\Omega f(x)\varphi(x)\,dx
$$
for every $\varphi\in\mathcal{D}(\Omega)$.
\end{definition}

We now state the main theorem (see \cite{B3} for details), which establishes the weak continuity of determinants.
\begin{theorem}\label{GTH2.4}
Let $\Omega \subset \mathbb{R}^n$ be a bounded open set, $1 < p < \infty$, and let  
$$
u_\nu \rightharpoonup u\quad\text{ in } W^{1,p}(\Omega; \mathbb{R}^n).
$$

(1) If $ p \geq n $, then
$$
\det \nabla u_\nu \rightharpoonup \det \nabla u \quad\text{ in } \mathcal{D}'(\Omega)
$$
and if $ p > n $, then
$$
\det \nabla u_\nu \rightharpoonup \det \nabla u \quad\text{in } L^{p/n}(\Omega).
$$

(2) Let $2 \leq s \leq n$ and $p \geq s $. Then
$$
\operatorname{adj}_s \nabla u_\nu \rightharpoonup \operatorname{adj}_s \nabla u \quad \text{in } \mathcal{D}' \left( \Omega; \mathbb{R}^{\sigma(s)} \right).
$$
Furthermore, if $p > s$, then
$$
\text{adj}_s \nabla u_\nu \rightharpoonup \text{adj}_s \nabla u \quad \text{in } L^{p/s} \left( \Omega; \mathbb{R}^{\sigma(s)} \right).
$$
Here $\sigma(s):=\binom{n}{s}^2$.
\end{theorem}

The following lemma extends the weak continuity of
Jacobians to continuous weights depending on the maps.
It will be used in the lower semicontinuity and
closedness arguments below.
\begin{lemma}[Weak continuity of weighted Jacobians]
\label{L2.1}
Let $n\ge2$ and let $\Omega\subset\mathbb R^n$ be a
bounded open set with Lipschitz boundary.
Suppose that
$v_k,v\in W^{1,n}(\Omega;\mathbb R^n)
\cap L^\infty(\Omega;\mathbb R^n)$ satisfy
$$
v_k\rightharpoonup v
\quad\text{in }W^{1,n}(\Omega;\mathbb R^n),
\qquad
v_k\rightharpoonup^*v
\quad\text{in }L^\infty(\Omega;\mathbb R^n).
$$
Then, for every $h\in C(\mathbb R^n)$,
$$
h(v_k)\det\nabla v_k
\rightharpoonup
h(v)\det\nabla v
\quad\text{in }\mathcal D'(\Omega).
$$
\end{lemma}

\begin{proof}
By compactness of the embedding
$W^{1,n}(\Omega)\hookrightarrow L^n(\Omega)$, we obtain that $
v_k\to v$ strongly in $L^n(\Omega;\mathbb R^n)$. The weak-$^*$ convergence in $L^\infty$ gives
a constant $R>0$ such that $
|v_k(x)|\le R$ and $|v(x)|\le R
$ for almost every $x\in\Omega$ and every $k$. We first consider $h\in C^1(\mathbb R^n)$.
For $y=(y_1,\ldots,y_n)\in\mathbb R^n$, define
$$
H_h(y):=\int_0^{y_1}h(t,y_2,\ldots,y_n)\,dt\quad\text{and}\quad\Psi_h(y):=(H_h(y),y_2,\ldots,y_n).
$$
Then $\Psi_h\in C^1(\mathbb R^n;\mathbb R^n)$ and
$$
D\Psi_h(y)=
\begin{pmatrix}
h(y) & \partial_{y_2}H_h(y) & \cdots
     & \partial_{y_n}H_h(y)\\
0 & 1 & \cdots & 0\\
\vdots & \vdots & \ddots & \vdots\\
0 & 0 & \cdots & 1
\end{pmatrix}.
$$
Thus,
$$
\det D\Psi_h(y)=h(y).
$$
We next verify the weak convergence of the
compositions. Since $D\Psi_h$ is continuous,
$$
C_{h,R}:=
\sup_{y\in\overline{B_R}}|D\Psi_h(y)|<\infty.
$$
In particular, $\Psi_h$ is Lipschitz on
$\overline{B_R}$.
The Sobolev chain rule gives
$\nabla(\Psi_h\circ v_k)=
D\Psi_h(v_k)\nabla v_k$ a.e. in $\Omega$. The chain rule is applicable because the maps $v_k$
take values in $\overline{B_R}$, on which the
derivative of $\Psi_h$ is bounded. Equivalently,
one may first replace $\Psi_h$ by a globally
Lipschitz $C^1$ map agreeing with it on a
neighborhood of $\overline{B_R}$. It follows that
$$
\|\nabla(\Psi_h\circ v_k)\|_{L^n(\Omega)}
\le
C_{h,R}\|\nabla v_k\|_{L^n(\Omega)}.
$$
Also, $\Psi_h$ is bounded on $\overline{B_R}$, so
$$
\|\Psi_h\circ v_k\|_{L^n(\Omega)}
\le
|\Omega|^{1/n}
\sup_{y\in\overline{B_R}}|\Psi_h(y)|.
$$
Hence $\{\Psi_h\circ v_k\}$ is uniformly bounded in
$W^{1,n}(\Omega;\mathbb R^n)$. On the other hand, the Lipschitz continuity of
$\Psi_h$ on $\overline{B_R}$ yields
$$
\|\Psi_h(v_k)-\Psi_h(v)\|_{L^n(\Omega)}
\le
C_{h,R}\|v_k-v\|_{L^n(\Omega)}
\rightarrow0.
$$
Therefore, $
\Psi_h\circ v_k
\rightharpoonup
\Psi_h\circ v$ in $W^{1,n}(\Omega;\mathbb R^n)$. Indeed, every subsequence admits
a further subsequence converging weakly in
$W^{1,n}$. The strong $L^n$ convergence identifies
every such weak limit with $\Psi_h\circ v$.
This proves weak convergence of the whole sequence. Applying Theorem \ref{GTH2.4} to these compositions, we obtain that
$
\det\nabla(\Psi_h\circ v_k)
\rightharpoonup
\det\nabla(\Psi_h\circ v)$ in $\mathcal D'(\Omega)$. The chain rule and multiplicativity of the
determinant yield
$$
\det\nabla(\Psi_h\circ v_k)
=
\det D\Psi_h(v_k)\det\nabla v_k
=
h(v_k)\det\nabla v_k,
$$
and the same identity holds for $v$.
This proves the assertion for $h\in C^1(\mathbb R^n)$. Now let $h\in C(\mathbb R^n)$.
Choose $h_m\in C^\infty(\mathbb R^n)$ such that
$$
\|h_m-h\|_{C(\overline{B_R})}\rightarrow0.
$$
The uniform $W^{1,n}$ bound gives
$$
C_J:=
\sup_k\|\det\nabla v_k\|_{L^1(\Omega)}
+\|\det\nabla v\|_{L^1(\Omega)}
<\infty.
$$
For every $\varphi\in C_c^\infty(\Omega)$,
$$
\begin{aligned}
&\left|
\int_\Omega\varphi
\bigl[h(v_k)\det\nabla v_k
      -h(v)\det\nabla v\bigr]\,dx
\right|
\\
&\quad\le
\left|
\int_\Omega\varphi
\bigl[h_m(v_k)\det\nabla v_k
      -h_m(v)\det\nabla v\bigr]\,dx
\right|
\\
&\qquad\quad+
C_J\|\varphi\|_{L^\infty(\Omega)}
\|h_m-h\|_{C(\overline{B_R})}.
\end{aligned}
$$
For fixed $m$, the first term tends to zero as
$k\to\infty$ by the preceding argument, and then $m\to\infty$
proves the conclusion.
\end{proof}
\subsection{$\Gamma$-convergence}
We recall sequential $\Gamma$-convergence and its
consequences for minimization.
These results provide the variational framework
for the source-penalty limits studied in Section~\ref{sec5.2}.

\begin{definition}
Let $Y$ be a topological space and let
$g_k,\,g:Y\to\mathbb R\cup\{+\infty\}$.
We say that $g_k$ sequentially $\Gamma$-converges to $g$ in $Y$, and write
$$
g_k\xrightarrow{\Gamma}g,
$$
if for every $y\in Y$ the following two conditions hold:
\begin{enumerate}
\item[(i)]
For every sequence $y_k\to y$ in $Y$,
$$
g(y)\le\liminf_{k\to\infty}g_k(y_k).
$$
\item[(ii)]
There exists a sequence $y_k\to y$ in $Y$ such that
$$
\limsup_{k\to\infty}g_k(y_k)\le g(y).
$$
\end{enumerate}
\end{definition}

\begin{proposition}\label{P4.1}\cite{G1}
Let $Y$ be a topological space and let
$g_k,g:Y\to\mathbb R\cup\{+\infty\}$.
Assume that
\begin{enumerate}
\item[(i)] $g_k\xrightarrow{\Gamma}g$ in $Y$;
\item[(ii)] the family $\{g_k\}$ is equicoercive, in the sense that for
every $t\in\mathbb R$ there exists a sequentially compact set $K_t\subset Y$ such that
$$
\{g_k\le t\}\subset K_t\quad\text{for all }k.$$
\end{enumerate}
Then
$$
\inf_Y g=
\lim_{k\to\infty}\inf_Y g_k.$$
Moreover, if $y_k$ minimizes $g_k$ for every $k$, then every subsequential limit of $\{y_k\}$
is a minimizer of $g$.
\end{proposition}

\section{Existence of Minimizers for Bounded Targets}\label{sec3}
In this section, we prove Theorems~\ref{TH5.1}
and~\ref{GTH5.2}.
We retain the assumptions of the bounded-target
setting, together with the ambient space $Y$
and the topology $\tau$ introduced above.
We first establish sequential closedness of the
admissible classes and lower semicontinuity of
$E_\alpha$ on $\mathcal V^{\rho_1}$.
We then obtain minimizers on $\mathcal V_K^{\rho_1}$
using the compactness provided by the distortion
and Jacobian mass bounds. The distortion inequality $
|\nabla M|^n\le K\det\nabla M
$
controls local stretching relative to volume change
and restricts anisotropic deformation.
For an orientation-preserving bi-Lipschitz map,
the change-of-variables formula gives
$$
\int_\Omega\det\nabla M\,dx=|M(\Omega)|.
$$
In particular, this integral equals $|\Omega|$
when $M$ maps $\Omega$ onto itself.
Thus the Jacobian mass bound in
$\mathcal V_K^{\rho_1}$ is consistent with such
transport maps.

\subsection{Sequential closedness of the admissible class}
We first prove that the admissible classes $\mathcal V^{\rho_1}$ and $\mathcal V_{K}^{\rho_1}$ are closed with respect to the weak topology $\tau$.

\begin{lemma}\label{L3.1}
The sets $\mathcal V^{\rho_1}$ and
$\mathcal V_K^{\rho_1}$ are sequentially closed
with respect to $\tau$.
\end{lemma}

\begin{proof}
Let $(M_k,f_k)\in\mathcal V^{\rho_1}$ satisfy $
(M_k,f_k)\xrightarrow{\tau}(M,f)$. We first verify that $(M,f)\in\mathcal V^{\rho_1}$. Since $\Omega$ is convex, $\overline{\Omega}$ is a closed convex subset
of $\mathbb{R}^n$. We claim that
$$
L^\infty(\Omega;\overline{\Omega})
:=
\left\{
M\in L^\infty(\Omega;\mathbb{R}^n):
M(x)\in\overline{\Omega}\ \text{for a.e. }x\in\Omega
\right\}
$$
is weak-$^*$ closed in $L^\infty(\Omega;\mathbb{R}^n)$. Suppose, by contradiction, that
$E:=\{x\in\Omega:M(x)\notin\overline{\Omega}\}$
has positive measure. Since $\overline{\Omega}$ is a closed convex subset of $\mathbb{R}^n$, it can be represented as an intersection of countably
many closed half-spaces,
$$
\overline{\Omega}
=
\bigcap_{j=1}^\infty
\left\{
y\in\mathbb{R}^n:
\xi_j\cdot y\le c_j
\right\}
$$
for some $\xi_j\in\mathbb{R}^n$ and $c_j\in\mathbb{R}$. Hence
$$E=
\bigcup_{j=1}^\infty
\left\{
x\in\Omega:
\xi_j\cdot M(x)>c_j
\right\}.
$$
Since $|E|>0$, there exist $j_0$ and a measurable set
$E_0
:=
\left\{
x\in\Omega:
\xi_{j_0}\cdot M(x)>c_{j_0}
\right\}
$
with $|E_0|>0$. Set
$h:=\xi_{j_0}\chi_{E_0}\in L^1(\Omega;\mathbb{R}^n)$. By the weak-$^*$ convergence of $M_k$, $
\int_\Omega h\cdot M_k\,dx
\rightarrow
\int_\Omega h\cdot M\,dx$. Since $M_k(x)\in\overline{\Omega}$ for a.e. $x\in\Omega$, we have
$\xi_{j_0}\cdot M_k(x)\le c_{j_0}$ for a.e. $x\in E_0$. Consequently,
$$
\int_\Omega h\cdot M_k\,dx
=
\int_{E_0}\xi_{j_0}\cdot M_k\,dx
\le
c_{j_0}|E_0|
$$
for every $k$. On the other hand,
$$
\int_\Omega h\cdot M\,dx
=
\int_{E_0}\xi_{j_0}\cdot M\,dx
>
c_{j_0}|E_0|,
$$
which contradicts the weak-$^*$ convergence. Therefore
$M(x)\in\overline{\Omega}$ for a.e. $x\in\Omega$. Thus $M\in L^\infty(\Omega;\overline{\Omega}),
$ and $
M\in X(\Omega;\overline{\Omega})$. 
For $0\le j\le n$, let $\sigma_j(\nabla M_k)$
denote the sum of all principal $j\times j$ minors,
with $\sigma_0\equiv1$.
Then
$$
\det\bigl(s\nabla M_k+(1-s)I\bigr)
=
\sum_{j=0}^n
s^j(1-s)^{n-j}\sigma_j(\nabla M_k).
$$
Fix $t\in[0,1]$ and $\varphi\in C_c^\infty(\Omega)$.
By Fubini's theorem and the finite expansion above,
$$
\int_0^t f_k(s)
\int_\Omega
\det\bigl(s\nabla M_k+(1-s)I\bigr)
\varphi(x)\,dx\,ds=
\sum_{j=0}^n
\left(
\int_0^t f_k(s)s^j(1-s)^{n-j}\,ds
\right)
\left(
\int_\Omega\sigma_j(\nabla M_k)\varphi\,dx
\right).
$$
For every fixed $t\in[0,1]$ and $0\le j\le n$, $s^j(1-s)^{n-j}\in L^2(0,1)$.
Hence the weak convergence of $f_k$ gives
$$
\int_0^t f_k(s)s^j(1-s)^{n-j}\,ds
\rightarrow
\int_0^t f(s)s^j(1-s)^{n-j}\,ds.
$$
For $2\le j\le n$, the distributional continuity
of minors (Theorem \ref{GTH2.4}) yields
$
\sigma_j(\nabla M_k)
\rightharpoonup
\sigma_j(\nabla M)$ in $\mathcal D'(\Omega)$. For $j=1$, this follows directly from the weak
convergence of the gradients, while the term
$j=0$ is independent of $k$.
Consequently,
\begin{equation}\label{ex3.2}
\begin{aligned}
&\int_0^t f_k(s)
\int_\Omega
\det\bigl(s\nabla M_k+(1-s)I\bigr)\varphi(x)\,dx\,ds\\
&\qquad\longrightarrow
\int_0^t f(s)
\int_\Omega
\det\bigl(s\nabla M+(1-s)I\bigr)\varphi(x)\,dx\,ds.
\end{aligned}
\end{equation}
Hence, for every fixed $t\in[0,1]$,
$
R_t[M_k,f_k]\rightharpoonup R_t[M,f]$ in $\mathcal D'(\Omega)$.
For every nonnegative $\varphi\in C_c^\infty(\Omega)$,
we have
$$
\int_\Omega R_t[M,f]\varphi\,dx
=
\lim_{k\to\infty}
\int_\Omega R_t[M_k,f_k]\varphi\,dx
\ge0.
$$
Since $R_t[M,f]\in L^1(\Omega)$, it follows that $
R_t[M,f]\ge0$ a.e. in $\Omega$
for every fixed $t\in[0,1]$. It remains to verify that $\det\bigl(t\nabla M(x)+(1-t)I\bigr)\geq0$ a.e. in $\Omega$. For each fixed
$t\in[0,1]$, define
$$
g_k^t(x)
:=
\det\bigl(t\nabla M_k(x)+(1-t)I\bigr).
$$
By the definition of $\mathcal{V}^{\rho_1}$, $
g_k^t\geq0\text{ a.e. in }\Omega
$. Furthermore, by Theorem \ref{GTH2.4},
$$
g_k^t
\rightharpoonup
g^t
:=
\det\bigl(t\nabla M+(1-t)I\bigr)
\quad\text{in }\mathcal{D}'(\Omega).
$$
Therefore, for every nonnegative
$\varphi\in \mathcal{D}(\Omega)$,
$$
\int_\Omega g^t\varphi\,dx=
\lim_{k\to\infty}
\int_\Omega g_k^t\varphi\,dx
\geq 0.
$$
Thus $g^t$ is a positive distribution. Since
$g^t\in L^1_{\mathrm{loc}}(\Omega)$, we conclude that, for every fixed $t\in[0,1]$
$
\det\bigl(t\nabla M(x)+(1-t)I\bigr)\geq0$ for a.e. $x\in\Omega$
. It remains to prove that
$\rho_1(M_k)\det\nabla M_k
\rightharpoonup
\rho_1(M)\det\nabla M$ in $\mathcal D'(\Omega)$. Since $\rho_1\in C(\overline\Omega)$ and
$\overline\Omega$ is closed, choose
$h\in C(\mathbb R^n)$ such that $
h=\rho_1$ on $\overline\Omega$. The weighted-Jacobian continuity Lemma \ref{L2.1} gives
$h(M_k)\det\nabla M_k
\rightharpoonup
h(M)\det\nabla M$ in $\mathcal D'(\Omega)$. Since $M_k$ and $M$ take values in $\overline\Omega$
almost everywhere, this is precisely
$$
\rho_1(M_k)\det\nabla M_k
\rightharpoonup
\rho_1(M)\det\nabla M
\quad\text{in }\mathcal D'(\Omega).
$$
Taking $t=1$ in the source-term convergence \eqref{ex3.2}
established above and passing to the limit in
the terminal constraint, we obtain
$$
\int_\Omega\rho_1(M)\det\nabla M\,\varphi\,dx=
\int_\Omega\rho_0\varphi\,dx+
\int_0^1 f(s)
\int_\Omega
\det\bigl(s\nabla M+(1-s)I\bigr)
\varphi\,dx\,ds
$$
for every $\varphi\in C_c^\infty(\Omega)$.
These arguments prove that
$(M,f)\in\mathcal V^{\rho_1}$ and hence establish
the sequential closedness of $\mathcal V^{\rho_1}$.

Suppose now that
$(M_k,f_k)\in\mathcal V_K^{\rho_1}$.
It remains to verify the bounded-distortion inequality
and the Jacobian mass bound for the limit. First, from $
|\nabla M_k|^n\le K\det\nabla M_k
$ we obtain, for every nonnegative $\varphi\in C_c^\infty(\Omega)$,
$$
\int_\Omega \varphi |\nabla M_k|^n\,dx
\le
K\int_\Omega \varphi\det\nabla M_k\,dx.
$$
By the weak lower semicontinuity of the convex functional
$M\mapsto\int_\Omega\varphi|\nabla M|^n\,dx$, and the distributional convergence
$\det\nabla M_k\rightharpoonup\det\nabla M$ in $\mathcal D'(\Omega)$,
we obtain
$$
\int_\Omega \varphi |\nabla M|^n\,dx
\le
K\int_\Omega \varphi\det\nabla M\,dx.
$$
Since the above inequality holds for every nonnegative
$\varphi\in C_c^\infty(\Omega)$, $
|\nabla M|^n\le K\det\nabla M$ a.e. in $\Omega$. Finally, we verify the Jacobian mass bound. Choose $\chi_m\in C_c^\infty(\Omega)$ such that
$0\le\chi_m\le1$, $\chi_m(x)\rightarrow1$ for every $x\in\Omega$. For each fixed $m$, the distributional convergence
of the determinants gives
$$
\int_\Omega\chi_m\det\nabla M\,dx
=
\lim_{k\to\infty}
\int_\Omega\chi_m\det\nabla M_k\,dx.
$$
Since $\det\nabla M_k\ge0$ and $0\le\chi_m\le1$,
we obtain
$$
\int_\Omega\chi_m\det\nabla M\,dx
\le
\liminf_{k\to\infty}
\int_\Omega\det\nabla M_k\,dx
\le|\Omega|.
$$
Letting $m\to\infty$ and applying the dominated
convergence theorem, we conclude that
$$
\int_\Omega\det\nabla M\,dx
\le
\liminf_{k\to\infty}
\int_\Omega\det\nabla M_k\,dx
\le|\Omega|.
$$
Thus $(M,f)\in\mathcal V_K^{\rho_1}$, which proves
the sequential closedness of this class as well.
\end{proof}
\subsection{Lower semicontinuity}
In this section, we establish the weak lower semicontinuity of
the functional $E_\alpha$ on the admissible class $\mathcal V^{\rho_1}$.
The main difficulty arises from the determinant term in the energy,
which does not fit directly into the standard weak lower semicontinuity
theory for quasiconvex functionals \cite{B3}. We first establish the distributional continuity
of determinant-weighted expressions.
This provides convergence of the localized
transport terms.
The nonnegativity of the intermediate states,
together with an interior cutoff argument,
then yields lower semicontinuity over the
whole domain.

The following lemma is the main ingredient in the proof of the weak
lower semicontinuity of $E_\alpha$.
\begin{lemma}\label{L3.2}
Let $n\ge2$ and let $\Omega\subset\mathbb R^n$
be a bounded open set with Lipschitz boundary.
Suppose that
$u_k,u\in W^{1,n}(\Omega;\mathbb R^n)
\cap L^\infty(\Omega;\mathbb R^n)$ satisfy
$$
u_k\rightharpoonup u
\quad\text{in }W^{1,n}(\Omega;\mathbb R^n),
\qquad
u_k\rightharpoonup^*u
\quad\text{in }L^\infty(\Omega;\mathbb R^n).
$$
Then, for every fixed $s\in[0,1]$,
$$
|u_k|^2\det(I+s\nabla u_k)
\rightharpoonup
|u|^2\det(I+s\nabla u)
\quad\text{in }\mathcal D'(\Omega).
$$
\end{lemma}

\begin{proof}
By compactness of the embedding
$W^{1,n}(\Omega)\hookrightarrow L^n(\Omega)$, $
u_k\to u$ strongly in $L^n(\Omega;\mathbb R^n)$. Since $n\ge2$ and $\Omega$ is bounded, this also
gives strong convergence in $L^2$.
Thus the conclusion holds for $s=0$. Fix $s>0$ and set
$$
v_k(x):=x+s u_k(x),
\qquad
v(x):=x+s u(x),
$$
and
$$
g_k:=\det\nabla v_k=\det(I+s\nabla u_k),
\qquad
g:=\det\nabla v=\det(I+s\nabla u).
$$
Then
$$
v_k\rightharpoonup v
\quad\text{in }W^{1,n}(\Omega;\mathbb R^n),
\qquad
v_k\rightharpoonup^*v
\quad\text{in }L^\infty(\Omega;\mathbb R^n).
$$
Applying Lemma~\ref{L2.1}
with
$$
h(y)=1,\qquad h(y)=y_i\ (1\le i\le n),
\qquad h(y)=|y|^2,
$$
we obtain
$$
g_k\rightharpoonup g,
\qquad
v_k^i g_k\rightharpoonup v^i g,
\qquad
|v_k|^2g_k\rightharpoonup |v|^2g
\quad\text{in }\mathcal D'(\Omega).
$$
For every $\varphi\in C_c^\infty(\Omega)$,
the identity $u_k=(v_k-x)/s$ gives
$$
\int_\Omega\varphi|u_k|^2g_k\,dx
=\frac1{s^2}\Bigg(\int_\Omega\varphi|v_k|^2g_k\,dx-2\sum_{i=1}^n
\int_\Omega(x_i\varphi)v_k^i g_k\,dx+\int_\Omega(|x|^2\varphi)g_k\,dx
\Bigg).
$$
Since $\varphi$, $x_i\varphi$, and
$|x|^2\varphi$ belong to $C_c^\infty(\Omega)$,
we may pass to the limit in each term.
Using $u=(v-x)/s$, we conclude that
$$
\int_\Omega\varphi|u_k|^2g_k\,dx
\rightarrow
\int_\Omega\varphi|u|^2g\,dx.
$$
This proves the assertion.
\end{proof}

\begin{proof}[Proof of Theorem \ref{TH5.1}]
Let $\{(M_k,f_k)\}_{k=1}^{\infty}
\subset\mathcal V^{\rho_1}$ satisfy
$$
M_k\rightharpoonup M
\quad\text{in }W^{1,n}(\Omega;\mathbb R^n),
\qquad
M_k\rightharpoonup^*M
\quad\text{in }L^\infty(\Omega;\mathbb R^n),
$$
and
$$
f_k\rightharpoonup f
\quad\text{in }L^2(0,1).
$$
By Lemma~\ref{L3.1}, we have
$(M,f)\in\mathcal V^{\rho_1}$.
We prove that
$$
E_\alpha(M,f)
\le
\liminf_{k\to\infty}E_\alpha(M_k,f_k).$$
If the lower limit is infinite, there is nothing
to prove. Otherwise, passing to a subsequence
and relabeling, we may assume that
$E_\alpha(M_k,f_k)
\rightarrow
\liminf_{j\to\infty}E_\alpha(M_j,f_j)<\infty$. Set
$u_k:=M_k-x$ and $u:=M-x$. Then
$u_k\rightharpoonup u$ in $W^{1,n}(\Omega;\mathbb R^n)$ and $
u_k\rightharpoonup^*u$ in $L^\infty(\Omega;\mathbb R^n)$. By compactness of the embedding
$W^{1,n}(\Omega)\hookrightarrow L^n(\Omega)$,
and after passing to a further subsequence, $u_k\to u$ strongly in $L^n(\Omega;\mathbb R^n)$. The prescribed bounded range also gives
$$
\sup_k\|u_k\|_{L^\infty(\Omega)}
+\|u\|_{L^\infty(\Omega)}
\le C_2
$$
for some constant $C_2>0$. 
Since $R_t[M,f]\ge0$, we have $G(M,f)\ge0$. In the bounded-range setting, the determinant
estimate and $f\in L^2(0,1)$ justify Fubini's
theorem, which gives
$$
G(M,f)
=
\int_\Omega |u|^2\rho_0\,dx+
\int_0^1(1-s)f(s)
\int_\Omega
|u|^2\det(I+s\nabla u)\,dx\,ds.
$$
Consequently,
$$
E_\alpha(M,f)
=
G(M,f)+\frac1\alpha\|f\|_{L^2(0,1)}^2.$$
Choose $\chi_m\in C_c^\infty(\Omega)$ such that
$
0\le\chi_m\le\chi_{m+1}\le1$ and $
\chi_m(x)\to1$ for every $x\in\Omega$.
For each $m\in\mathbb N$, define
$$
G_m(M,f):=
\int_0^1\int_\Omega
\chi_m|u|^2R_t[M,f]\,dx\,dt.
$$
We first show that, for every fixed $m$, $G_m(M_k,f_k)\rightarrow G_m(M,f)$. For $s\in[0,1]$, set
$$
\Phi_{k,m}(s):=
\int_\Omega
\chi_m|u_k|^2\det(I+s\nabla u_k)\,dx
$$
and
$$
\Phi_m(s):=
\int_\Omega
\chi_m|u|^2\det(I+s\nabla u)\,dx.
$$
For every fixed $s\in[0,1]$, Lemma~\ref{L3.2} gives
$
|u_k|^2\det(I+s\nabla u_k)
\rightharpoonup
|u|^2\det(I+s\nabla u)$ in $\mathcal D'(\Omega)$. Testing this convergence against $\chi_m$, we obtain that
$\Phi_{k,m}(s)\rightarrow\Phi_m(s)$ for every $s\in[0,1]$. We next establish uniform convergence in $s$.
The weak convergence implies that
$\sup_k\|\nabla u_k\|_{L^n(\Omega)}<\infty$. Using
$
|\det(I+sA)|\le C_n(1+|A|^n)$ for $s\in[0,1]$ and $A\in\mathbb{R}^{n\times n}$, we obtain
$$
\sup_k\sup_{s\in[0,1]}
\int_\Omega
|\det(I+s\nabla u_k)|\,dx
\le C_3.
$$
Since $0\le\chi_m\le1$ and the interpolation
Jacobians are nonnegative, we obtain that $
0\le\Phi_{k,m}(s)\le C_2^2C_3$. Moreover, the polynomial structure of the
determinant gives
$|\det(I+sA)-\det(I+rA)|
\le C_n|s-r|(1+|A|^n)$. Therefore,
$$
\begin{aligned}
|\Phi_{k,m}(s)-\Phi_{k,m}(r)|
&\le
C_2^2C_n|s-r|
\int_\Omega(1+|\nabla u_k|^n)\,dx\\
&\le C_5|s-r|,
\end{aligned}
$$
where $C_5$ is independent of $k$ and $m$.
Thus, for fixed $m$, the sequence
$\{\Phi_{k,m}\}_k$ is uniformly bounded and
equicontinuous on $[0,1]$. By the Arzel\`a-Ascoli theorem and the pointwise
convergence established above,
$\Phi_{k,m}\rightarrow\Phi_m$ uniformly on $[0,1]$. Since $\{f_k\}$ is bounded in $L^2(0,1)$,
$$
\begin{aligned}
&\left|
\int_0^1(1-s)f_k(s)
\bigl(\Phi_{k,m}(s)-\Phi_m(s)\bigr)\,ds
\right|\\
&\qquad\le
\|f_k\|_{L^2(0,1)}
\|\Phi_{k,m}-\Phi_m\|_{L^2(0,1)}
\rightarrow0.
\end{aligned}
$$
Also, weak convergence of $f_k$ gives
$$
\int_0^1(1-s)f_k(s)\Phi_m(s)\,ds
\rightarrow
\int_0^1(1-s)f(s)\Phi_m(s)\,ds.
$$
Combining these two convergences, we obtain
$$
\int_0^1(1-s)f_k(s)\Phi_{k,m}(s)\,ds
\rightarrow
\int_0^1(1-s)f(s)\Phi_m(s)\,ds.
$$
Furthermore, $u_k\to u$ almost everywhere and
$0\le\chi_m|u_k|^2\rho_0\le C_2^2\rho_0
\in L^1(\Omega)$. The dominated convergence theorem therefore yields
$$
\int_\Omega\chi_m|u_k|^2\rho_0\,dx
\rightarrow
\int_\Omega\chi_m|u|^2\rho_0\,dx.
$$
By Fubini's theorem,
$$
G_m(M_k,f_k)
=
\int_\Omega\chi_m|u_k|^2\rho_0\,dx
+
\int_0^1(1-s)f_k(s)\Phi_{k,m}(s)\,ds.
$$
Hence $G_m(M_k,f_k)\rightarrow G_m(M,f)$. We now remove the cutoff.
Since $R_t[M_k,f_k]\ge0$ and $0\le\chi_m\le1$, $
G(M_k,f_k)\ge G_m(M_k,f_k)$.
For every fixed $m$, it follows that
$$
\liminf_{k\to\infty}G(M_k,f_k)
\ge G_m(M,f).
$$
Since $R_t[M,f]\ge0$ and $\chi_m\rightarrow1$,
the monotone convergence theorem gives
$G_m(M,f)\rightarrow G(M,f)$. Consequently,
$$
G(M,f)
\le
\liminf_{k\to\infty}G(M_k,f_k).
$$
Finally, weak lower semicontinuity of the
$L^2$ norm yields
$$
\frac1\alpha\|f\|_{L^2(0,1)}^2
\le
\liminf_{k\to\infty}
\frac1\alpha\|f_k\|_{L^2(0,1)}^2.
$$
Combining the last two inequalities and using
the nonnegativity of both terms, we conclude that
$$
E_\alpha(M,f)
\le
\liminf_{k\to\infty}E_\alpha(M_k,f_k).
$$
Thus $E_\alpha$ is sequentially weakly lower
semicontinuous on $\mathcal V^{\rho_1}$.
\end{proof}

\vspace{3mm}

\subsection{Existence of minimizers}
We now prove the existence of minimizers in $\mathcal V_K^{\rho_1}$.
The bounded-distortion condition yields the coercive estimate
$$
\|\nabla M\|_{L^n(\Omega)}^n
\le
K\int_\Omega\det\nabla M\,dx
\le K|\Omega|
$$
for every $(M,f)\in\mathcal V_K^{\rho_1}$.

\begin{proof}[Proof of Theorem \ref{GTH5.2}] 
Let
$$
m:=
\inf_{(M,f)\in\mathcal V_K^{\rho_1}}E_\alpha(M,f).
$$
Since $\mathcal V_K^{\rho_1}\ne\varnothing$,
choose $(M_0,f_0)\in\mathcal V_K^{\rho_1}$.
The bounded range of $M_0$, the determinant
estimate, and $f_0\in L^2(0,1)$ ensure that $
E_\alpha(M_0,f_0)<\infty$. Moreover, the nonnegativity of $G(M,f)$ gives
$$
E_\alpha(M,f)
=
G(M,f)+\frac1\alpha\|f\|_{L^2(0,1)}^2
\ge0
$$
for every $(M,f)\in\mathcal V_K^{\rho_1}$.
Thus $
0\le m\le E_\alpha(M_0,f_0)<\infty$.
Choose a minimizing sequence
$(M_k,f_k)\in\mathcal V_K^{\rho_1}$ such that $
E_\alpha(M_k,f_k)\to m$. We may assume that
$E_\alpha(M_k,f_k)\le E_\alpha(M_0,f_0)$ for every $k$. Since $G(M_k,f_k)\ge0$,
$$
\frac1\alpha\|f_k\|_{L^2(0,1)}^2
\le E_\alpha(M_k,f_k)
\le E_\alpha(M_0,f_0),
$$
$\{f_k\}$ is bounded in $L^2(0,1)$. The bounded-distortion condition and the
Jacobian mass bound imply
$$
\|\nabla M_k\|_{L^n(\Omega)}^n
\le
K\int_\Omega\det\nabla M_k\,dx
\le K|\Omega|.
$$
Since $M_k(x)\in\overline\Omega$ almost everywhere
and $\Omega$ is bounded, we also have
$\sup_k\|M_k\|_{L^\infty(\Omega)}<\infty$. Consequently, $\{M_k\}$ is bounded in both
$W^{1,n}(\Omega;\mathbb R^n)$ and
$L^\infty(\Omega;\mathbb R^n)$. By weak compactness and weak-$^*$ sequential
compactness, after passing to a subsequence,
$$
M_k\rightharpoonup M
\quad\text{in }W^{1,n}(\Omega;\mathbb R^n),
$$
$$
M_k\rightharpoonup^*M
\quad\text{in }L^\infty(\Omega;\mathbb R^n),
\qquad
f_k\rightharpoonup f
\quad\text{in }L^2(0,1).
$$
By Lemma~\ref{L3.1}, $(M,f)\in\mathcal V_K^{\rho_1}$.
The lower semicontinuity established in Theorem~\ref{TH5.1} gives
$$
E_\alpha(M,f)
\le
\liminf_{k\to\infty}E_\alpha(M_k,f_k)
=m.
$$
Since $(M,f)$ is admissible, the definition of
$m$ gives the reverse inequality.
Thus $E_\alpha(M,f)=m$, which proves the existence of a minimizer.
\end{proof}

\section{Existence of Minimizers for Whole-Space Targets}\label{sec4}

We extend the preceding results to nonnegative target
densities $\rho_1\in C(\mathbb R^n)\cap L^1(\mathbb R^n)$,
allowing the admissible maps to take values in
$\mathbb R^n$.
Throughout this section, $\Omega$ is a bounded connected
Lipschitz domain,
$
0\le\rho_0\in L^1(\Omega)$ and $
0<\int_\Omega\rho_0\,dx$.
The prescribed bounded range is replaced by local
compactness under a fixed bounded-distortion condition.
We also replace the Jacobian bound $|\Omega|$ by
an independent constant $\Lambda>0$.
Together, these constraints give
\begin{equation}\label{e5.1}
\|\nabla M\|_{L^n(\Omega)}^n
\le K\int_\Omega\det\nabla M\,dx
\le K\Lambda.
\end{equation}

\begin{lemma}\label{L5.2}
The set $\mathcal V_{K,\Lambda}^{\rho_1}$ is sequentially closed with respect
to the convergence
$$
M_k\rightharpoonup M
\quad\text{in }W^{1,n}(\Omega;\mathbb R^n)\quad\text
{and}\quad
f_k\rightharpoonup f
\quad\text{in }L^2(0,1).
$$
Moreover, if $(M_k,f_k)\in\mathcal V_{K,\Lambda}^{\rho_1}$ converges
weakly to $(M,f)$, after passing to a
subsequence,
$$
M_k\to M
\quad\text{locally uniformly in }\Omega.
$$
\end{lemma}

\begin{proof}
The sequence $\{M_k\}$ is bounded in $W^{1,n}$ and
satisfies a fixed bounded-distortion inequality.
The quantitative local H\"older estimates in
\cite[Lemmas~5.1 and~5.2]{OP} control the local
H\"older seminorms by the $L^n$ norm of the gradient.
Combined with the uniform $L^n$ bound for the maps
themselves, these estimates imply that, for every
$U\Subset\Omega$,
$$
\sup_k\|M_k\|_{C^{0,\beta}(\overline U)}<\infty,
$$
where $\beta\in(0,1)$ depends only on the fixed
distortion bound $K$.
Consequently, the Arzel\`a-Ascoli theorem and a
diagonal argument yield a subsequence converging
locally uniformly in $\Omega$. The nonnegativity of
$\det(t\nabla M+(1-t)I)$ and $R_{t}[M,f]$, and the two additional inequalities defining
$\mathcal V_{K,\Lambda}^{\rho_1}$ follow as in
Lemma~\ref{L3.1}, with $|\Omega|$
replaced by $\Lambda$. It only remains to pass to the limit in the weighted Jacobian $\rho_1(M_{k})\det\nabla M_{k}$. Fix $\varphi\in C_c^\infty(\Omega)$.
Local uniform convergence of $M_{k}$ and the continuity of $\rho_1$
imply $
\rho_1(M_k)\to\rho_1(M)$ uniformly on $\operatorname{supp}\varphi$. On the other hand, the distributional convergence of
$\det\nabla M_k$, together with its uniform $L^1$ bound,
implies convergence against every function in
$C_c(\Omega)$, by the weak-$^*$ convergence of measure. Consequently,
$$
\begin{aligned}
\int_\Omega\varphi
&\bigl[\rho_1(M_k)\det\nabla M_k
-\rho_1(M)\det\nabla M\bigr]\,dx\\
&=
\int_\Omega\varphi
\bigl[\rho_1(M_k)-\rho_1(M)\bigr]
\det\nabla M_k\,dx+
\int_\Omega\varphi\rho_1(M)
\bigl[\det\nabla M_k-\det\nabla M\bigr]\,dx
\rightarrow0.
\end{aligned}
$$
Indeed, the first term is controlled by uniform
convergence and the bound $\Lambda$, while the second
uses the continuous test function $\varphi\rho_1(M)\in C_{c}(\Omega)$. The source term passes to the limit in $\mathcal{D}'(\Omega)$ by the same finite expansion in
principal minors as in Lemma~\ref{L3.1}.
Thus the limiting pair satisfies the constraint and
belongs to $\mathcal V_{K,\Lambda}^{\rho_1}$.
\end{proof}

\begin{proof}[Proof of Theorem \ref{T5.3}]
Let $(M_k,f_k)\in\mathcal V_{K,\Lambda}^{\rho_1}$
converge weakly to $(M,f)$ in $Y_1$.
By Lemma~\ref{L5.2},
$(M,f)\in\mathcal V_{K,\Lambda}^{\rho_1}$.
We prove that
$$
E_\alpha(M,f)
\le\liminf_{k\to\infty}E_\alpha(M_k,f_k).
$$
If the lower limit is infinite, there is nothing
to prove. Otherwise, we pass to a subsequence
realizing this finite lower limit.
By Lemma~\ref{L5.2}, after passing to a further
subsequence, we have
$M_k\to M$ locally uniformly in $\Omega$. Set
$u_k:=M_k-x$ and $u:=M-x$. Recall that
$$
E_\alpha(M,f)
=
G(M,f)+\frac1\alpha\|f\|_{L^2(0,1)}^2,
$$
where 
$$
G(M,f)
=
\int_0^1\int_\Omega
|u|^2R_t[M,f]\,dx\,dt
$$
is nonnegative.
Since the source penalty is weakly lower
semicontinuous in $L^2(0,1)$, it remains to
treat $G$. Choose $\chi_m\in C_c^\infty(\Omega)$ such that
$
0\le\chi_m\le\chi_{m+1}\le1$ and 
$\chi_m(x)\to1$ for every $x\in\Omega$, and define
$$
G_m(M,f)
:=
\int_0^1\int_\Omega
\chi_m|u|^2R_t[M,f]\,dx\,dt.
$$
Fix $m$ and choose a bounded Lipschitz open set
$U_m$ such that
$\operatorname{supp}\chi_m\Subset U_m\Subset\Omega$. Then
$
u_k\rightharpoonup u$ in $W^{1,n}(U_m;\mathbb R^n)$, and 
$u_k\to u$ uniformly on $\overline{U_m}$. In particular,
$$
\sup_k\|u_k\|_{L^\infty(U_m)}
+\|u\|_{L^\infty(U_m)}
\le C_m<\infty.
$$
Thus Lemma~\ref{L3.2}, applied on $U_m$, gives
$$
\int_\Omega
\chi_m|u_k|^2\det(I+s\nabla u_k)\,dx
\rightarrow
\int_\Omega
\chi_m|u|^2\det(I+s\nabla u)\,dx
$$
for every $s\in[0,1]$.
As in the proof of Theorem~\ref{TH5.1},
the polynomial dependence on $s$ and the uniform
gradient bounds imply that this convergence
is uniform on $[0,1]$.
Moreover, local uniform convergence yields
$$
\int_\Omega\chi_m|u_k|^2\rho_0\,dx
\rightarrow
\int_\Omega\chi_m|u|^2\rho_0\,dx.
$$
Using Fubini's theorem and
$f_k\rightharpoonup f$ in $L^2(0,1)$,
the same localized argument gives
$
G_m(M_k,f_k)\rightarrow G_m(M,f)$.
Since $R_t[M_k,f_k]\ge0$ and $0\le\chi_m\le1$, we have 
$G(M_k,f_k)\ge G_m(M_k,f_k)$. Thus, for every fixed $m$
$$
\liminf_{k\to\infty}G(M_k,f_k)\ge G_m(M,f).
$$
Letting $m\to\infty$ and using
$R_t[M,f]\ge0$, the monotone convergence theorem
gives
$$
G(M,f)\le\liminf_{k\to\infty}G(M_k,f_k).
$$
Combining this inequality with the weak lower
semicontinuity of the source penalty, we obtain
$$
E_\alpha(M,f)
\le
\liminf_{k\to\infty}G(M_k,f_k)
+\liminf_{k\to\infty}
\frac1\alpha\|f_k\|_{L^2(0,1)}^2\le
\liminf_{k\to\infty}E_\alpha(M_k,f_k).
$$
Since the subsequence was chosen to realize
the original lower limit, this proves the
sequential weak lower semicontinuity of
$E_\alpha$ on $\mathcal V_{K,\Lambda}^{\rho_1}$.

We next establish a coercive estimate on energy
sublevel sets. Let $(M,f)$ be admissible with
$E_\alpha(M,f)<\infty$. Because $\int_\Omega\rho_0\,dx>0$, there exists $c>0$
such that
$D:=\{x\in\Omega:\rho_0(x)\ge c\}$
has positive measure. Choose
$$
Q:=\max\left\{
1,\left(\frac{2K\Lambda}{|D|}\right)^{1/n}
\right\}.
$$
This choice depends only on the fixed data and is
independent of the admissible pair $(M,f)$.
For each admissible map $M$, set
$$
D_M:=D\cap\{|\nabla M|\le Q\}.
$$
Since $D\setminus D_M=D\cap\{|\nabla M|>Q\}$, the gradient bound gives
$$
Q^n|D\setminus D_M|
\le
\int_{D\setminus D_M}|\nabla M|^n\,dx
\le
\int_\Omega|\nabla M|^n\,dx
\le K\Lambda.
$$
By the choice of $Q$, it follows that
$$
|D_M|
=|D|-|D\setminus D_M|
\ge |D|-\frac{K\Lambda}{Q^n}
\ge\frac{|D|}{2}.
$$
On $D_M$, for every $s\in[0,1]$,
$$
0\le \det(s\nabla M+(1-s)I)\le(1+Q)^n=:C_Q.
$$
In particular, $C_Q$ is independent of $M$ and $f$,
although the set $D_M$ depends on $M$.
Consequently,
$R_t[M,f]
\ge c-C_Q\sqrt{t}\,\|f\|_{L^2(0,1)}$ a.e. on $D_M$. Set
$$
\delta:=
\min\left\{
1,\,
\frac{c^2}
{4C_Q^2(1+\|f\|_{L^2(0,1)})^2}
\right\}.
$$
Then, for $0\le t\le\delta$, $
R_t[M,f]\ge\frac c2$ a.e. on $D_M$. Since the $R_t[M,f]$ are nonnegative,
$$
G(M,f)
\ge
\int_0^\delta\int_{D_M}
|M-x|^2R_t[M,f]\,dx\,dt
\ge
\frac{c\delta}{2}
\int_{D_M}|M-x|^2\,dx.
$$
It follows that
$$
\|M\|_{L^2(D_M)}
\le
C\left(
1+(1+\|f\|_{L^2(0,1)})G(M,f)^{1/2}
\right).
$$
Let
$$
\overline M:=\frac1{|\Omega|}\int_\Omega M\,dx.
$$
Since $|D_M|\ge |D|/2$, Poincar\'e's inequality gives
$$
\begin{aligned}
|\overline M|\,|D_M|^{1/2}
&\le
\|M\|_{L^2(D_M)}
+\|M-\overline M\|_{L^2(D_M)}\\
&\le
\|M\|_{L^2(D_M)}
+C\|\nabla M\|_{L^n(\Omega)}.
\end{aligned}
$$
Using \eqref{e5.1} and Poincar\'e's inequality again,
we obtain
\begin{equation}\label{e5.2}
\|M\|_{W^{1,n}(\Omega)}
\le
C\left(
1+(1+\|f\|_{L^2(0,1)})G(M,f)^{1/2}
\right),
\end{equation}
where $C$ depends only on $\Omega$, $n$, $\rho_0$,
$K$, and $\Lambda$. The finite-energy competitor ensures that the infimum is finite. Thus a minimizing sequence
can be chosen with uniformly bounded energy.
Since
$$
G(M_k,f_k)\le E_\alpha(M_k,f_k),
\qquad
\|f_k\|_{L^2(0,1)}^2
\le\alpha E_\alpha(M_k,f_k),
$$
estimate \eqref{e5.2} gives a uniform
$W^{1,n}$ bound for $M_k$.
Weak compactness, the sequential closedness of
the admissible class, and the lower semicontinuity
proved above yield an admissible minimizer.
\end{proof}

\section{Static-Dynamic Equivalence and Source-Penalty Limits}\label{sec5}
This section gives a dynamic interpretation of the
static problem under prescribed affine-flow conditions.
We establish an energy-preserving correspondence and
its stability, which transfers the convergence of
static minimizers to reconstructed dynamic paths.
We then study strong source penalization and identify
the role of constant-source feasibility in determining
the leading optimal cost.

\subsection{Static-dynamic equivalence for affine flows}
\label{sec5.1}

We first prove the static-dynamic correspondence
and establish stability of the reconstruction. Throughout this subsection, let $n\ge2$ and let
$\Omega\subset\mathbb R^n$ be a bounded convex open set with
Lipschitz boundary. Assume that
$0\le\rho_0\in L^1(\Omega)$ and
$0\le\rho_1\in C(\overline\Omega)$.
We retain the admissible class $\mathcal V^{\rho_1}$ and the
topology $\tau$ defined above.  Write
$$
u:=M-\mathrm{Id},\qquad
X_t[M](x):=x+tu(x),\qquad
J_t[M](x):=\det(I+t\nabla u(x)),
$$
and recall that
$$
R_t[M,f](x)
:=\rho_0(x)+\int_0^t f(s)J_s[M](x)\,ds,
$$
$$
E_\alpha(M,f)
=\int_0^1\int_\Omega |u(x)|^2R_t[M,f](x)\,dx\,dt
+\frac1\alpha\int_0^1|f(t)|^2\,dt.
$$

\begin{lemma}\label{L6.1}
The set $\mathcal A_L^{\rho_1}$ is sequentially closed with
respect to $\tau$. For every $(M,f)\in\mathcal A_L^{\rho_1}$
and every $t\in[0,1]$, $X_t[M]$ is an
orientation-preserving bi-Lipschitz homeomorphism of $\Omega$
onto itself. Moreover,
\begin{equation}\label{ex5.2}
L^{-n}\le J_t[M]\le L^n
\qquad\text{for all }t\in[0,1],
\end{equation}
and
\begin{equation}\label{ex5.3}
\int_\Omega J_t[M]\,dx=|\Omega|
\qquad\text{for all }t\in[0,1].
\end{equation}
Using the Frobenius norm for matrices, one also has
\begin{equation}\label{eq:reconstruction-distortion}
|\nabla X_t[M]|^n
\le n^{n/2}L^{2n}J_t[M]
\qquad\text{a.e. in }\Omega.
\end{equation}
In particular,
$$
|\nabla M|^n
\le n^{n/2}L^{2n}\det\nabla M\text{ a.e. in }\Omega\quad\text{and}\quad
\int_\Omega\det\nabla M\,dx=|\Omega|.
$$
\end{lemma}

\begin{proof}
Let $(M_j,f_j)\in\mathcal A_L^{\rho_1}$ converge to $(M,f)$
with respect to $\tau$. The sequential closedness of
$\mathcal V^{\rho_1}$ established by Lemma \ref{L3.1} gives
$(M,f)\in\mathcal V^{\rho_1}$.
Taking $t=1$ in \eqref{ex1.3}, we see that
the $M_j$ are uniformly $L$-Lipschitz and take values in the
compact set $\overline\Omega$. By the Arzel\`a-Ascoli theorem,
a subsequence converges uniformly on $\overline\Omega$.
Its uniform limit agrees almost everywhere with the weak
limit $M$, providing a continuous representative of $M$.
Consequently,
$$
X_t[M_j]\rightarrow X_t[M]
\quad\text{uniformly on }[0,1]\times\overline\Omega.
$$
To verify surjectivity, fix $t\in[0,1]$ and
$y\in\overline\Omega$. Choose $x_j\in\overline\Omega$ with
$X_t[M_j](x_j)=y$. After taking a further subsequence,
$x_j\to x\in\overline\Omega$, and uniform convergence yields
$X_t[M](x)=y$. Thus $(M,f)\in\mathcal A_L^{\rho_1}$. For any map satisfying
\eqref{ex1.3},
each $X_t[M]$ is a bi-Lipschitz homeomorphism
of $\overline\Omega$ onto itself.
Such a homeomorphism preserves the interior and
the boundary, so $X_t[M](\Omega)=\Omega$. By Rademacher's theorem, $M$ is differentiable
outside a single null set.
At every such point, the distance bounds imply that,
for every $t\in[0,1]$, all singular values of
$I+t\nabla u$ lie in $[L^{-1},L]$.
Consequently,
$$
L^{-n}\le
\bigl|\det(I+t\nabla u)\bigr|
\le L^n.
$$
The determinant is continuous in $t$, never vanishes,
and equals one at $t=0$.
It therefore remains positive on $[0,1]$. There exists a null set $N\subset\Omega$ such that
$L^{-n}\le J_t[M](x)\le L^n$
for every $x\in\Omega\setminus N$ and every $t\in[0,1]$. This also proves orientation preservation. The change-of-variables
formula then yields
$$
\int_\Omega J_t[M]\,dx
=|X_t[M](\Omega)|
=|\Omega|,
$$
which proves \eqref{ex5.3}. The singular-value upper bound also gives
$|\nabla X_t[M]|\le\sqrt n\,L$. Combining this with
$J_t[M]\ge L^{-n}$ yields
$$
|\nabla X_t[M]|^n
\le n^{n/2}L^n
\le n^{n/2}L^{2n}J_t[M].
$$
Taking $t=1$ proves the assertions concerning $M$.
\end{proof}

Writing $Y_t:=X_t[M]^{-1}$ and
$D:=\operatorname{diam}(\Omega)$, we have
\begin{equation}\label{ex5.5}
|Y_t(y)-Y_s(z)|
\le L|y-z|+LD|t-s|.
\end{equation}
Indeed, the spatial estimate follows from
\eqref{ex1.3}, while
$$
|Y_t(y)-Y_s(y)|
\le L|X_t[M](Y_s(y))-X_s[M](Y_s(y))|
\le LD|t-s|.
$$
In particular, the inverse flow is jointly Lipschitz. Let $\mathfrak D^{\rho_0,\rho_1}$ consist of triples
$(\rho,v,f)$ such that
$$
0\le\rho\in C([0,1];L^1(\Omega)),\qquad
v\in L^\infty((0,1)\times\Omega;\mathbb R^n),\qquad
f\in L^2(0,1),
$$
and
\begin{equation}\label{ex5.6}
\int_\Omega\rho_1(y)\varphi(1,y)\,dy
-\int_\Omega\rho_0(y)\varphi(0,y)\,dy=
\int_0^1\int_\Omega
\left(
\rho\,\partial_t\varphi+
\rho v\cdot\nabla\varphi+
f(t)\varphi
\right)\,dy\,dt
\end{equation}
for every $\varphi\in C^1([0,1]\times\overline\Omega)$.
This identity expresses the continuity equation, the
prescribed endpoint conditions, and the no-flux boundary
condition. Define
$$
\mathscr E_\alpha(\rho,v,f)
:=
\int_0^1\int_\Omega\rho(t,y)|v(t,y)|^2\,dy\,dt
+\frac1\alpha\int_0^1|f(t)|^2\,dt.
$$
The dynamic problem considered here is posed on
$\mathfrak D^{\rho_0,\rho_1}$, with the stated
bounded-velocity and strong $L^1$ continuity requirements.
The correspondence below concerns the subclass whose
velocities generate uniformly bi-Lipschitz affine flows.
Let $\mathfrak D_L^{\rho_0,\rho_1}$ consist of triples in
$\mathfrak D^{\rho_0,\rho_1}$ for which $v$ has a jointly
Lipschitz representative on $[0,1]\times\overline\Omega$
whose flow $X_t$ satisfies
$$
\frac{d}{dt}X_t(x)=v(t,X_t(x)),\qquad X_0(x)=x,
$$
together with
\begin{equation}\label{ex5.7}
X_t(x)=(1-t)x+tX_1(x),
\end{equation}
and
$$
X_t(\overline\Omega)=\overline\Omega,\qquad
L^{-1}|x-y|\le |X_t(x)-X_t(y)|\le L|x-y|
$$
for all $x,y\in\overline\Omega$ and $t\in[0,1]$.

\begin{proof}[Proof of Theorem \ref{TH1.4}]
We first prove existence. Every pair in $\mathcal A_L^{\rho_1}$ has finite energy, since
$|M-\mathrm{Id}|\le D$, $J_t\le L^n$, and
$$
\|R_t[M,f]\|_{L^1(\Omega)}
\le \|\rho_0\|_{L^1(\Omega)}
+L^n|\Omega|\|f\|_{L^1(0,1)}.
$$
Moreover, $R_t[M,f]\ge0$, so
$E_\alpha(M,f)\ge\alpha^{-1}\|f\|_{L^2(0,1)}^2$.
For a minimizing sequence, the sources are therefore
bounded in $L^2(0,1)$, while the maps are uniformly
Lipschitz with values in $\overline\Omega$.
After passing to a subsequence, the pairs converge with
respect to $\tau$, and the maps also converge uniformly.
By Lemma \ref{L6.1}, the limit belongs to $\mathcal A_L^{\rho_1}$.
The lower semicontinuity of $E_\alpha$ on
$\mathcal V^{\rho_1}$, together with the sequential
closedness of $\mathcal A_L^{\rho_1}$, shows that
the limit attains the minimum.
No additional bounded-distortion or Jacobian mass
constraint is required, since
Lemma~\ref{L6.1} shows that the geometric conditions
\eqref{ex1.3} already imply
$$
|\nabla M|^n
\le n^{n/2}L^{2n}\det\nabla M
\quad\text{a.e. in }\Omega\quad\text{and}\quad\int_\Omega\det\nabla M\,dx=|\Omega|.
$$
Fix $(M,f)\in\mathcal A_L^{\rho_1}$ and abbreviate
$X_t=X_t[M]$, $J_t=J_t[M]$, and $R_t=R_t[M,f]$.
Define $\rho$ and $v$ by \eqref{eq:dynamic-reconstruction}.
Since $R_t\ge0$ and $J_t\ge L^{-n}>0$, the
reconstructed density satisfies $\rho\ge0$.
Moreover, $M(x),x\in\overline\Omega$ imply $
|u(x)|=|M(x)-x|\le D$, and hence $|v(t,y)|\le D$. Taking $t=1$ in  \eqref{ex1.3}, we see that
$M$ is $L$-Lipschitz. Therefore,
$$
|u(x)-u(z)|
\le (L+1)|x-z|
\quad
\text{for all }x,z\in\overline\Omega.
$$
Together with \eqref{ex5.5}, this gives
$$
\begin{aligned}
|v(t,y)-v(s,z)|
&=|u(Y_t(y))-u(Y_s(z))|\\
&\le (L+1)|Y_t(y)-Y_s(z)|\\
&\le L(L+1)|y-z|
   +LD(L+1)|t-s|.
\end{aligned}
$$
Thus $v$ is jointly Lipschitz on
$[0,1]\times\overline\Omega$. Finally, the definition $X_t(x)=x+tu(x)$ yields
$$
\partial_tX_t(x)
=u(x)
=u(Y_t(X_t(x)))
=v(t,X_t(x)),
\qquad
X_0(x)=x.
$$
Hence $X_t$ solves the flow equation associated
with $v$. Since $X_1=M$, we have $X_t(x)=(1-t)x+tX_1(x)$,
which is precisely \eqref{ex5.7}. We verify strong continuity of the density.
The polynomial dependence of $J_t$ on $t$, together
with $\nabla M\in L^\infty$, gives
$J_t\in C([0,1];L^\infty(\Omega))$.
By \eqref{ex5.3},
$$
\|R_t-R_s\|_{L^1(\Omega)}
\le |\Omega|
\int_{\min\{s,t\}}^{\max\{s,t\}}|f(r)|\,dr.
$$
Thus $R_t$ is continuous in $L^1$.
Fix $s\in[0,1]$ and choose
$h\in C(\overline\Omega)$ with
$\|\rho(s,\cdot)-h\|_{L^1}<\varepsilon$.
The change-of-variables formula yields
$$
\|\rho(t,\cdot)-h\|_{L^1(\Omega)}
=
\|R_t-h(X_t)J_t\|_{L^1(\Omega)}.
$$
The right-hand side is continuous in $t$.
Hence
$$
\limsup_{t\to s}
\|\rho(t,\cdot)-\rho(s,\cdot)\|_{L^1(\Omega)}
\le 2\|\rho(s,\cdot)-h\|_{L^1(\Omega)}
<2\varepsilon.
$$
This proves $\rho\in C([0,1];L^1(\Omega))$. Let $\varphi\in C^1([0,1]\times\overline\Omega)$ and set
$$
F(t):=\int_\Omega R_t(x)\varphi(t,X_t(x))\,dx.
$$
Since $\partial_tR_t=f(t)J_t$ is integrable and
$\varphi(t,X_t(x))$ has a bounded time derivative,
$F$ is absolutely continuous.
For almost every $t\in(0,1)$,
$$
\begin{aligned}
F'(t)
={}&\int_\Omega f(t)J_t(x)\varphi(t,X_t(x))\,dx\\
&+\int_\Omega R_t(x)
\left[
\partial_t\varphi(t,X_t(x))
+u(x)\cdot\nabla\varphi(t,X_t(x))
\right]\,dx.
\end{aligned}
$$
Integrating this identity over $t\in(0,1)$ gives
$$
\begin{aligned}
F(1)-F(0)
={}&\int_0^1\int_\Omega
f(t)J_t(x)\varphi(t,X_t(x))\,dx\,dt\\
&+\int_0^1\int_\Omega R_t(x)
\left[
\partial_t\varphi(t,X_t(x))
+u(x)\cdot\nabla\varphi(t,X_t(x))
\right]\,dx\,dt.
\end{aligned}
$$
For each $t$, make the change of variables $y=X_t(x)$.
Using $
dy=J_t(x)\,dx$, $R_t(x)=\rho(t,X_t(x))J_t(x)$, $u(x)=v(t,X_t(x))$,
and $X_t(\Omega)=\Omega$, we obtain
$$
F(1)-F(0)
=
\int_0^1\int_\Omega
\left[
\rho\,\partial_t\varphi
+\rho v\cdot\nabla\varphi
+f(t)\varphi
\right]\,dy\,dt.
$$
Finally, $X_0=\mathrm{Id}$ and $R_0=\rho_0$ yield
$$
F(0)=\int_\Omega\rho_0(y)\varphi(0,y)\,dy.
$$
Since $X_1=M$ and
$R_1=\rho_1(M)\det\nabla M$, another change of variables gives
$$
\begin{aligned}
F(1)
&=\int_\Omega
\rho_1(M(x))\det\nabla M(x)\,
\varphi(1,M(x))\,dx\\
&=\int_\Omega\rho_1(y)\varphi(1,y)\,dy.
\end{aligned}
$$
Substituting these endpoint identities proves \eqref{ex5.6}. Thus $(\rho,v,f)\in\mathfrak D_L^{\rho_0,\rho_1}$.

We next prove the converse.
Let $(\rho,v,f)\in\mathfrak D_L^{\rho_0,\rho_1}$ and let
$X_t$ be its flow. Set $M=X_1$.
By \eqref{ex5.7},
$$
X_t=X_t[M],\qquad
v(t,X_t(x))=M(x)-x.
$$
The geometric argument in Lemma~\ref{L6.1} gives
$J_t[M]>0$ a.e. for every $t\in[0,1]$. We justify the density identity directly from the weak
equation. For $\eta\in C_c^1(0,1)$ and
$\zeta\in C^1(\overline\Omega)$, set
$$
\varphi(t,y):=\eta(t)\zeta(Y_t(y)).
$$
Although $\varphi$ is only Lipschitz, it is an
admissible test function by smooth approximation.
Indeed, the approximations converge uniformly,
with uniformly bounded derivatives converging
almost everywhere, so the weak identity passes
to the limit by dominated convergence. The space-time map
$(t,x)\mapsto(t,X_t(x))$ and its inverse are
bi-Lipschitz. Applying the Lipschitz chain rule to
$Y_t(X_t(x))=x$ gives
$$
\partial_tY_t+D_yY_t\,v=0
\quad\text{a.e. in }(0,1)\times\Omega.
$$
Thus
$$
\partial_t\varphi+v\cdot\nabla\varphi=\eta'(t)\zeta(Y_t(y)).
$$
It follows that
$$
\frac{d}{dt}
\int_\Omega\rho(t,y)\zeta(Y_t(y))\,dy
=
f(t)\int_\Omega\zeta(Y_t(y))\,dy
=
f(t)\int_\Omega\zeta(x)J_t(x)\,dx
$$
in distributions in time.
The left-hand integral is continuous in $t$, and the
right-hand side is integrable.
For every $t\in[0,1]$,
$$
\int_\Omega\rho(t,X_t(x))J_t(x)\zeta(x)\,dx
=
\int_\Omega
\left(\rho_0(x)+\int_0^t f(s)J_s(x)\,ds\right)
\zeta(x)\,dx.
$$
Since $\zeta$ is arbitrary, we obtain
$\rho(t,X_t(x))J_t(x)=R_t[M,f](x)$ for a.e. $x$. In particular, $R_t[M,f]\ge0$ for every fixed $t$.
The terminal condition gives
$$
\rho_1(M)\det\nabla M
=
\rho_0+\int_0^1 f(s)J_s[M]\,ds.
$$
Thus $(M,f)\in\mathcal V^{\rho_1}$ and, by the flow
geometry, $(M,f)\in\mathcal A_L^{\rho_1}$. Finally, in either direction the change-of-variables
formula gives
$$
\begin{aligned}
\mathscr E_\alpha(\rho,v,f)
&=\int_0^1\int_\Omega
|M(x)-x|^2R_t[M,f](x)\,dx\,dt
+\frac1\alpha\int_0^1|f(t)|^2\,dt\\
&=E_\alpha(M,f).
\end{aligned}
$$
The two constructions and static existence prove
\eqref{eq:static-dynamic-equality}, attainment of the
dynamic minimum, and the assertions about minimizers.
\end{proof}

\begin{corollary}\label{cor:full-dynamic-equivalence}
Assume that the minimum of $\mathscr E_\alpha$ on
$\mathfrak D^{\rho_0,\rho_1}$ is attained by a triple
belonging to $\mathfrak D_L^{\rho_0,\rho_1}$.
Then
\[
\min_{(\rho,v,f)\in\mathfrak D^{\rho_0,\rho_1}}
\mathscr E_\alpha(\rho,v,f)
=
\min_{(M,f)\in\mathcal A_L^{\rho_1}}E_\alpha(M,f).
\]
Every minimizer on $\mathcal A_L^{\rho_1}$ reconstructs
to a minimizer on $\mathfrak D^{\rho_0,\rho_1}$.
\end{corollary}

\begin{proof}
The assumed dynamic minimizer yields a pair in
$\mathcal A_L^{\rho_1}$ by
Theorem~\ref{TH1.4}.
Since
$\mathfrak D_L^{\rho_0,\rho_1}\subset
\mathfrak D^{\rho_0,\rho_1}$ and this subclass contains
a minimizer of the larger problem, the two dynamic
minimum values coincide. The conclusion now follows
from \eqref{eq:static-dynamic-equality}.
\end{proof}

\begin{remark}
Gangbo et al.~\cite{bib1} derive the affine interpolation
from the Hamilton-Jacobi optimality equation.
However, their analysis does not establish the
domain-preserving bi-Lipschitz flow properties required
here, nor the prescribed uniform bound $L$. Thus equality between the minimum of $E_\alpha$ on
$\mathcal A_L^{\rho_1}$ and the minimum of
$\mathscr E_\alpha$ on $\mathfrak D^{\rho_0,\rho_1}$
requires the additional hypothesis of
Corollary~\ref{cor:full-dynamic-equivalence}.
\end{remark}

\begin{proposition}[Stability of the reconstruction]
\label{P5.1}
Let $(M_j,f_j),(M,f)\in\mathcal A_L^{\rho_1}$ satisfy
$$
M_j\to M
\quad\text{uniformly on }\overline\Omega,
\qquad
f_j\rightharpoonup f
\quad\text{in }L^2(0,1).
$$
Let $(\rho_j,v_j,f_j)$ and $(\rho,v,f)$ be their
reconstructed dynamic triples.
Then $
v_j\to v$ uniformly on $[0,1]\times\overline\Omega$, and, for every $\psi\in C(\overline\Omega)$,
$$
\sup_{t\in[0,1]}
\left|
\int_\Omega
\psi(y)\bigl(\rho_j(t,y)-\rho(t,y)\bigr)\,dy
\right|
\rightarrow0.
$$
Moreover, $
\rho_j\rightharpoonup^*\rho$ in $L^\infty((0,1)\times\Omega)$ and 
$G(M_j,f_j)\to G(M,f)$.
\end{proposition}
\begin{proof}
Write $
X_{j,t}:=X_t[M_j]$, $X_t:=X_t[M]$,
$Y_{j,t}:=X_{j,t}^{-1}$, $Y_t:=X_t^{-1}$,
and set $
u_j:=M_j-\mathrm{Id}$, $u:=M-\mathrm{Id}$. We first prove convergence of the velocities.
The inverse Lipschitz bound gives
$$
\begin{aligned}
|Y_{j,t}(y)-Y_t(y)|
&\le
L|X_{j,t}(Y_{j,t}(y))-X_{j,t}(Y_t(y))|\\
&=
L|X_t(Y_t(y))-X_{j,t}(Y_t(y))|\\
&\le L\|M_j-M\|_{L^\infty(\overline\Omega)}.
\end{aligned}
$$
Since $u$ is $(L+1)$-Lipschitz and $v_j(t,y)=u_j(Y_{j,t}(y))$, $v(t,y)=u(Y_t(y))$, we obtain
$$
\|v_j-v\|_{L^\infty([0,1]\times\overline\Omega)}
\le
\bigl(1+L(L+1)\bigr)
\|M_j-M\|_{L^\infty(\overline\Omega)}
\rightarrow0.
$$
We next establish a uniform bound for the densities.
The endpoint constraint for $(M,f)$ and the Jacobian
bounds imply
$$
0\le\rho_0
=
\rho_1(M)J_1[M]
-
\int_0^1f(s)J_s[M]\,ds
\le
L^n\left(
\|\rho_1\|_{L^\infty(\Omega)}
+\|f\|_{L^1(0,1)}
\right).
$$
Thus $\rho_0\in L^\infty(\Omega)$.
Since weak convergence of $(f_j)$ gives $
B:=\sup_j\|f_j\|_{L^2(0,1)}<\infty$, the reconstruction formulas yield
$$
0\le R_t[M_j,f_j]
\le \|\rho_0\|_{L^\infty(\Omega)}+L^nB
$$
and
$$
\sup_j\sup_{t\in[0,1]}
\|\rho_j(t,\cdot)\|_{L^\infty(\Omega)}
\le
L^n\|\rho_0\|_{L^\infty(\Omega)}+L^{2n}B.
$$
The same type of bound holds for $\rho$. To prove convergence of the intermediate states, expand
$$
J_t[M_j]=\sum_{k=0}^n c_{j,k}t^k,
\qquad
J_t[M]=\sum_{k=0}^n c_kt^k.
$$
The coefficients are sums of minors of $\nabla u_j$
and $\nabla u$, respectively.
The uniform Lipschitz bounds and uniform convergence
of the maps imply that $
M_j\rightharpoonup M$ in $W^{1,p}(\Omega;\mathbb R^n)$ for every $1<p<\infty$. By weak continuity of minors and the uniform
$L^\infty$ bounds on these coefficients, $
c_{j,k}\rightharpoonup^*c_k$ in $L^\infty(\Omega)$, $k=0,\ldots,n$. Define
$$
F_{j,k}(t):=\int_0^t f_j(s)s^k\,ds,
\qquad
F_k(t):=\int_0^t f(s)s^k\,ds.
$$
Weak convergence of the sources gives
$F_{j,k}(t)\to F_k(t)$ for every $t$. Moreover, $
|F_{j,k}(t)-F_{j,k}(s)|
\le B|t-s|^{1/2}$. Consequently,
$
F_{j,k}\to F_k$ uniformly on $[0,1]$. Writing $
R_{j,t}:=R_t[M_j,f_j]$ and $R_t:=R_t[M,f]$, we therefore have
$$
R_{j,t}=\rho_0+\sum_{k=0}^n c_{j,k}F_{j,k}(t),
\qquad
R_t=\rho_0+\sum_{k=0}^n c_kF_k(t).
$$
It follows that, for every $h\in L^1(\Omega)$,
\begin{equation}\label{eq:intermediate-state-stability}
\sup_{t\in[0,1]}
\left|
\int_\Omega h(x)(R_{j,t}(x)-R_t(x))\,dx
\right|
\longrightarrow0.
\end{equation}
Indeed, each term can be estimated using the uniform
convergence of $F_{j,k}$ and the weak-$^*$ convergence of $c_{j,k}$. Fix $\psi\in C(\overline\Omega)$.
By change of variables,
$$
\begin{aligned}
&\int_\Omega
\psi(y)(\rho_j(t,y)-\rho(t,y))\,dy\\
&\quad=
\int_\Omega
\bigl(\psi(X_{j,t}(x))-\psi(X_t(x))\bigr)
R_{j,t}(x)\,dx\\
&\qquad+
\int_\Omega
\psi(X_t(x))(R_{j,t}(x)-R_t(x))\,dx.
\end{aligned}
$$
The first term tends to zero uniformly in $t$,
because $X_{j,t}\to X_t$ uniformly in space and time
and the intermediate states are uniformly bounded. For the second term, since $\{\psi\circ X_t:t\in[0,1]\}$ is compact in
$L^1(\Omega)$ and $R_{j,t}-R_t$ is uniformly bounded
in $L^\infty(\Omega)$, a finite approximation argument
using \eqref{eq:intermediate-state-stability} gives
$$
\sup_{t\in[0,1]}
\left|
\int_\Omega
\psi(X_t(x))(R_{j,t}(x)-R_t(x))\,dx
\right|
\rightarrow0.$$
Hence
$$
\sup_{t\in[0,1]}
\left|
\int_\Omega
\psi(y)(\rho_j(t,y)-\rho(t,y))\,dy
\right|
\rightarrow0.$$
This convergence implies convergence against
space-time test functions of the form
$\eta(t)\psi(y)$, with
$\eta\in L^1(0,1)$ and
$\psi\in C(\overline\Omega)$.
Finite sums of such functions are dense in
$L^1((0,1)\times\Omega)$.
Together with the uniform density bounds, this proves $
\rho_j\rightharpoonup^*\rho$ in $L^\infty((0,1)\times\Omega)$. Finally, the energy-preserving reconstruction gives
$$
G(M_j,f_j)
=
\int_0^1\int_\Omega\rho_j|v_j|^2\,dy\,dt,
\qquad
G(M,f)
=
\int_0^1\int_\Omega\rho|v|^2\,dy\,dt.
$$
Thus
$$
\begin{aligned}
G(M_j,f_j)-G(M,f)
={}&
\int_0^1\int_\Omega
\rho_j\bigl(|v_j|^2-|v|^2\bigr)\,dy\,dt\\
&+
\int_0^1\int_\Omega
(\rho_j-\rho)|v|^2\,dy\,dt.
\end{aligned}
$$
The first term tends to zero by uniform convergence
of the velocities and the uniform density bounds.
The second tends to zero by weak-$^*$ convergence of the  densities. Therefore $G(M_j,f_j)\to G(M,f)$.
\end{proof}

\subsection{Strong source penalization and the constant-source limit}
\label{sec5.2}
We work in the bounded-target setting and use
$\mathcal A_L^{\rho_1}$.
As $\alpha\to0^+$, minimization first selects the
smallest admissible source cost and then the least
transport energy among the resulting pairs.
The endpoint mass difference fixes the mean source.
The question is whether this mean can be realized
by a constant admissible source.
We establish the general limit first and then
characterize the constant-source case.

Set
$$
a:=
\frac{\displaystyle\int_\Omega\rho_1(x)\,dx
-\displaystyle\int_\Omega\rho_0(x)\,dx}{|\Omega|}.
$$
Every $(M,f)\in\mathcal A_L^{\rho_1}$ determines
a dynamic admissible triple by the reconstruction
established above.
Taking $\varphi\equiv1$ in the weak formulation of
the continuity equation gives
$$
\int_\Omega\rho_1(y)\,dy-\int_\Omega\rho_0(y)\,dy
=|\Omega|\int_0^1f(t)\,dt.
$$
Thus $\int_0^1f(t)\,dt=a$, and consequently
$$
\|f\|_{L^2(0,1)}^2
=a^2+\|f-a\|_{L^2(0,1)}^2.
$$
In particular,
\begin{equation}\label{ex5.10}
E_\alpha(M,f)-\frac{a^2}{\alpha}
=G(M,f)+\frac1\alpha\|f-a\|_{L^2(0,1)}^2.
\end{equation}
The source cost is therefore bounded below by $a^2$,
with equality precisely when $f=a$ a.e.
Whether this lower bound is attained in
$\mathcal A_L^{\rho_1}$ also depends on the endpoint
densities and the prescribed geometric constraints. $b_L$ is the least source cost compatible with
the admissibility conditions.

\begin{proof}[Proof of Theorem \ref{thm:minimum-source-limit}]
Since $\mathcal A_L^{\rho_1}\ne\varnothing$ and every
admissible source belongs to $L^2(0,1)$, we have $0\le b_L<\infty$.
Choose a minimizing sequence
$(M_j,f_j)\in\mathcal A_L^{\rho_1}$ such that
$\|f_j\|_{L^2(0,1)}^2\rightarrow b_L$.
In particular, $(f_j)$ is bounded in $L^2(0,1)$. Since $\Omega$ is bounded and
$M_j(\overline\Omega)\subset\overline\Omega$,
the maps are uniformly bounded in $L^\infty(\Omega)$.
Their common Lipschitz bound also gives $
\|\nabla M_j\|_{L^\infty(\Omega)}
\le \sqrt n\,L$. Thus $(M_j)$ is bounded in
$W^{1,n}(\Omega;\mathbb R^n)$, weak compactness in $W^{1,n}$ and
$L^2(0,1)$ yields a subsequence such that
$$
M_j\rightharpoonup M
\quad\text{in }W^{1,n}(\Omega;\mathbb R^n),
\qquad
f_j\rightharpoonup f
\quad\text{in }L^2(0,1).
$$
Moreover, the maps are uniformly bounded and
equicontinuous on the compact set $\overline\Omega$.
The Arzel\`a--Ascoli theorem therefore gives a further
subsequence converging uniformly on $\overline\Omega$.
The uniform limit agrees almost everywhere with the
weak Sobolev limit $M$, and we use this continuous
representative of $M$.
Hence $
M_j\to M$ uniformly on $\overline\Omega$.
Finally, the sequential closedness established in
Lemma~\ref{L6.1} implies
$(M,f)\in\mathcal A_L^{\rho_1}$.
Applying this compactness to a minimizing sequence for
$b_L$, and using weak lower semicontinuity of the
$L^2$ norm, proves that $b_L$ is attained.
The same argument shows that $\mathcal S_L^{\rho_1}$
is sequentially $\tau$-compact. 
Lower semicontinuity of $G$ therefore gives a minimizer
on $\mathcal S_L^{\rho_1}$.
Since every admissible pair has finite transport energy,
$0\le g_L<\infty$.

For the liminf inequality, let $\alpha_j\to0^+$ and
$(M_j,f_j)\xrightarrow{\tau}(M,f)$.
If the lower limit of
$\widehat{\mathcal F}_{\alpha_j}(M_j,f_j)$ is infinite,
there is nothing to prove.
Otherwise, pass to a subsequence realizing the finite
lower limit, along which these energies are bounded.
The pairs then belong to $\mathcal A_L^{\rho_1}$, and
$$
0\le\|f_j\|_{L^2(0,1)}^2-b_L
\le\alpha_j\widehat{\mathcal F}_{\alpha_j}(M_j,f_j)
\rightarrow0.
$$
Sequential closedness and weak lower semicontinuity give
$(M,f)\in\mathcal A_L^{\rho_1}$ and
$$b_L\le\|f\|_{L^2(0,1)}^2
\le\liminf_{j\to\infty}\|f_j\|_{L^2(0,1)}^2=b_L.
$$
Thus $(M,f)\in\mathcal S_L^{\rho_1}$, and
$$
\liminf_{j\to\infty}
\widehat{\mathcal F}_{\alpha_j}(M_j,f_j)
\ge\liminf_{j\to\infty}G(M_j,f_j)
\ge G(M,f)=\widehat{\mathcal F}_0(M,f).
$$
For $(M,f)\in\mathcal S_L^{\rho_1}$, the constant sequence
is a recovery sequence because
$$
\widehat{\mathcal F}_\alpha(M,f)
=G(M,f)=\widehat{\mathcal F}_0(M,f).
$$
Outside this set, the recovery inequality is automatic.
This proves sequential $\Gamma$-convergence.

If $0<\alpha\le1$ and
$\widehat{\mathcal F}_\alpha(M,f)\le T$, then
$$
\|f\|_{L^2(0,1)}^2\le b_L+\alpha T\le b_L+T.
$$
Together with the map bounds and sequential closedness,
this gives sequential equicoercivity.
Let $(\widehat M,\widehat f)\in\mathcal S_L^{\rho_1}$
satisfy $G(\widehat M,\widehat f)=g_L$.
By minimality of $(M_\alpha,f_\alpha)$ and
$\|\widehat f\|_{L^2(0,1)}^2=b_L$, we have
\begin{equation}\label{ex5.13}
0\le G(M_\alpha,f_\alpha)\le \widehat{\mathcal F}_\alpha(M_\alpha,f_\alpha)\le \widehat{\mathcal F}_\alpha(\widehat M,\widehat f)
=g_L.
\end{equation}
Consequently,
$$
0\le \|f_\alpha\|_{L^2(0,1)}^2-b_L
\le \alpha g_L.
$$
Thus the sources are bounded in $L^2(0,1)$ for
$0<\alpha\le1$, and their squared norms converge to $b_L$. For any sequence $\alpha_j\to0^+$, the compactness
established above yields a subsequence such that
$$
(M_{\alpha_j},f_{\alpha_j})\xrightarrow{\tau}(M,f),
\qquad
M_{\alpha_j}\to M
\quad\text{uniformly on }\overline\Omega.
$$
Sequential closedness gives
$(M,f)\in\mathcal A_L^{\rho_1}$.
By the definition of $b_L$ and weak lower
semicontinuity,
$$
b_L\le \|f\|_{L^2(0,1)}^2
\le \liminf_{j\to\infty}
\|f_{\alpha_j}\|_{L^2(0,1)}^2=b_L.
$$
Hence $(M,f)\in\mathcal S_L^{\rho_1}$.
Moreover, weak convergence of the sources and
convergence of their norms imply
$f_{\alpha_j}\to f$ strongly in $L^2(0,1)$. Since $g_L$ is the minimum of $G$ on
$\mathcal S_L^{\rho_1}$, lower semicontinuity and
\eqref{ex5.13} give
$$
g_L\le G(M,f)
\le \liminf_{j\to\infty}G(M_{\alpha_j},f_{\alpha_j})
\le \limsup_{j\to\infty}G(M_{\alpha_j},f_{\alpha_j})
\le g_L.
$$
Thus $G(M,f)=g_L$ and
$G(M_{\alpha_j},f_{\alpha_j})\to g_L$.
Using \eqref{ex5.13} once more,
we obtain $
\widehat{\mathcal F}_{\alpha_j}
(M_{\alpha_j},f_{\alpha_j})\to g_L$
and
$$
\frac{\|f_{\alpha_j}\|_{L^2(0,1)}^2-b_L}{\alpha_j}
=
\widehat{\mathcal F}_{\alpha_j}
(M_{\alpha_j},f_{\alpha_j})
-G(M_{\alpha_j},f_{\alpha_j})
\rightarrow0.
$$
Since every sequence $\alpha_j\to0^+$ admits a
subsequence along which the transport energies
converge to $g_L$ and the rescaled source penalties
converge to zero, both limits in
\eqref{eq:minimum-source-energy-limits} hold as
$\alpha\to0^+$ for any choice of minimizers.
Finally,
$$
\min_{\mathcal A_L^{\rho_1}}E_\alpha-\frac{b_L}{\alpha}
=
G(M_\alpha,f_\alpha)
+\frac{\|f_\alpha\|_{L^2(0,1)}^2-b_L}{\alpha}
\rightarrow g_L,
$$
which proves \eqref{eq:minimum-source-expansion}.
\end{proof}

\begin{corollary}[Convergence of reconstructed paths]
Let $\alpha_j\to0^+$ and let
$(M_{\alpha_j},f_{\alpha_j})$ be minimizers as in
Theorem~\ref{thm:minimum-source-limit}.
Along the subsequence provided by that theorem,
let $(\rho_{\alpha_j},v_{\alpha_j},f_{\alpha_j})$
and $(\rho,v,f)$ be the reconstructed triples.
Then $
v_{\alpha_j}\to v$ uniformly on $[0,1]\times\overline\Omega$,
and, for every $\psi\in C(\overline\Omega)$,
$$
\sup_{t\in[0,1]}
\left|
\int_\Omega
\psi(y)(\rho_{\alpha_j}(t,y)-\rho(t,y))\,dy
\right|\longrightarrow0.
$$
Moreover, $\rho_{\alpha_j}\rightharpoonup^*\rho$
in $L^\infty((0,1)\times\Omega)$, and
$(\rho,v,f)$ minimizes
$$
\int_0^1\int_\Omega\rho|v|^2\,dy\,dt
$$
among triples in $\mathfrak D_L^{\rho_0,\rho_1}$
satisfying $\|f\|_{L^2(0,1)}^2=b_L$.
\end{corollary}

\begin{proof}
The convergence follows from
Proposition~\ref{P5.1}.
The minimizing property follows from
$G(M,f)=g_L$ and the energy-preserving correspondence.
\end{proof}

We now identify when the preceding limit is a
constant-source problem.
The mass constraint gives $b_L\ge a^2$.
Since $b_L$ is attained, equality holds if and only if
there exists $(M_*,a)\in\mathcal A_L^{\rho_1}$.
In that case, $
\mathcal S_L^{\rho_1}
=\{(M,a)\in\mathcal A_L^{\rho_1}\}$. Thus the constant-source feasibility assumption means
that the lower bound imposed by the total mass difference
can be attained while satisfying the endpoint and
geometric constraints. Retaining the notation for the constant-source problem,
define
$$
\mathcal F_{\alpha,a}(M,f):=
\begin{cases}
E_\alpha(M,f)-a^2/\alpha,
&(M,f)\in\mathcal A_L^{\rho_1},\\
+\infty,&\text{otherwise},
\end{cases}
$$
and
$$
\mathcal F_{0,a}(M,f):=
\begin{cases}
G(M,a),&f=a\ \text{a.e. and }(M,a)\in\mathcal A_L^{\rho_1},\\
+\infty,&\text{otherwise}.
\end{cases}
$$

\begin{corollary}[Constant-source limit]
\label{thm:constant-source-limit}
Assume that there exists $(M_*,a)\in\mathcal A_L^{\rho_1}$.
Then
$$
\mathcal F_{\alpha,a}
\xrightarrow{\Gamma}\mathcal F_{0,a}
\qquad\text{as }\alpha\to0^+
$$
in the sequential sense with respect to $\tau$.
The constant-source problem admits a minimizer, and,
writing
$$
m_a:=\min\{G(M,a):(M,a)\in\mathcal A_L^{\rho_1}\},
$$
we have
\begin{equation}\label{eq:constant-source-expansion}
\min_{\mathcal A_L^{\rho_1}}E_\alpha
=\frac{a^2}{\alpha}+m_a+o(1).
\end{equation}
For every sequence $\alpha_j\to0^+$ and corresponding
minimizers $(M_{\alpha_j},f_{\alpha_j})$, there exist a
subsequence, not relabeled, and a minimizer $(M,a)$ of
$\mathcal F_{0,a}$ such that
$$
M_{\alpha_j}\rightharpoonup M
\quad\text{in }W^{1,n}(\Omega;\mathbb R^n),
\qquad
M_{\alpha_j}\to M
\quad\text{uniformly on }\overline\Omega,
$$
and
$$
f_{\alpha_j}\to a
\quad\text{strongly in }L^2(0,1).
$$
Moreover, for any choice of minimizers,
\begin{equation}\label{ex5.11}
G(M_\alpha,f_\alpha)\to m_a,
\qquad
\frac1\alpha\|f_\alpha-a\|_{L^2(0,1)}^2\to0.
\end{equation}
In particular, $
\|f_\alpha-a\|_{L^2(0,1)}=o(\sqrt\alpha)$.
\end{corollary}

\begin{proof}
The assumption gives $b_L=a^2$ and $g_L=m_a$.
Hence $\widehat{\mathcal F}_\alpha=\mathcal F_{\alpha,a}$
and $\widehat{\mathcal F}_0=\mathcal F_{0,a}$.
All convergence assertions follow from
Theorem~\ref{thm:minimum-source-limit} and the identity
$\|f\|_{L^2(0,1)}^2-a^2=\|f-a\|_{L^2(0,1)}^2$.
In particular, comparison with a constant-source
minimizer $(\widehat M,a)$ gives
\begin{equation}\label{ex5.12}
0\le G(M_\alpha,f_\alpha)
\le\mathcal F_{\alpha,a}(M_\alpha,f_\alpha)
\le\mathcal F_{\alpha,a}(\widehat M,a)=m_a.
\end{equation}
Both middle terms converge to $m_a$, so their difference
is $\alpha^{-1}\|f_\alpha-a\|_{L^2(0,1)}^2\to0$.
\end{proof}

\begin{remark}
The limit selects the least transport energy among
admissible pairs with constant source $a$.
The dynamic paths reconstructed from the limiting
pairs create or remove mass uniformly in space
at a constant rate in time, so their total mass
varies linearly between the endpoint masses,
$$
\int_\Omega\rho(t,y)\,dy
=(1-t)\int_\Omega\rho_0(y)\,dy
+t\int_\Omega\rho_1(y)\,dy.
$$
The limiting problem minimizes the transport energy
needed to connect the prescribed densities under
this source.
When $a=0$, it reduces to the zero-source transport
problem within $\mathcal A_L^{\rho_1}$.
\end{remark}

\begin{proposition}[A lower bound for the excess source cost]
\label{P5.2}
Assume that $\mathcal A_L^{\rho_1}\ne\varnothing$ and
$a<0$. Set
$$
c_0:=\operatorname*{ess\,inf}_{\Omega}\rho_0,
\qquad
c_1:=\min_{\overline\Omega}\rho_1.$$
Then
$$
b_L-a^2
\ge
L^{-2n}
\left[
L^{-n}(c_1+|a|)-c_0
\right]_+^2,
$$
where $[r]_+:=\max\{r,0\}$.
In particular, a necessary condition for the existence
of a constant-source admissible pair is
$$
c_0\ge L^{-n}(c_1+|a|).
$$
\end{proposition}

\begin{proof}
Let $(M,f)\in\mathcal A_L^{\rho_1}$ and write $h:=f-a$.
The endpoint constraint gives
$$
\rho_0=
\rho_1(M)J_1[M]
+
|a|\int_0^1J_t[M]\,dt
-
\int_0^1h(t)J_t[M]\,dt.
$$
Using $L^{-n}\le J_t[M]\le L^n$ and
Cauchy-Schwarz, we obtain
$$
\rho_0
\ge
L^{-n}(c_1+|a|)
-
L^n\|h\|_{L^2(0,1)}
\quad\text{a.e. in }\Omega.
$$
Taking the essential infimum yields
$$
\|f-a\|_{L^2(0,1)}
\ge
L^{-n}
\left[
L^{-n}(c_1+|a|)-c_0
\right]_+.
$$
Since
$$
\|f\|_{L^2(0,1)}^2
=
a^2+\|f-a\|_{L^2(0,1)}^2,
$$
taking the infimum over $\mathcal A_L^{\rho_1}$
proves the assertion.
\end{proof}

When the lower bound in Proposition~\ref{P5.2}
is positive, the prescribed densities and flow bounds
force every admissible source to deviate from its mean.
Consequently,
$$
\min_{\mathcal A_L^{\rho_1}}E_\alpha
-\frac{a^2}{\alpha}
=
\frac{b_L-a^2}{\alpha}+g_L+o(1)
\rightarrow+\infty.
$$
Thus the total mass difference alone does not determine
the leading optimal cost.
Example~\ref{ex:no-constant-source} shows that this
phenomenon occurs even when the reconstruction class
is nonempty.

\section{Nonemptiness of Admissible Classes}\label{sec6}

\subsection{A local sufficient condition for nonemptiness}
\label{subsec:local-feasibility}
The following result establishes local feasibility
when both endpoint densities are H\"older continuous. The construction applies to perturbations of a fixed
reference pair $(\overline\rho-a,\overline\rho)$.
The reference density need not be constant, and the
distortion and flow bounds are prescribed in advance.

\begin{proposition}[Local feasibility for H\"older densities]
\label{prop:local-feasibility}
Let $n\ge2$, $K>n^{n/2}$, $L>1$ and $0<\beta<\gamma<1$, and let
$\Omega\subset\mathbb R^n$ be a bounded convex domain
with $C^{3,\beta}$ boundary.
Fix
$\overline\rho\in C^{1,\gamma}(\overline\Omega)$
and $a\in\mathbb R$ such that $\overline\rho\ge0$, $\overline\rho-a\ge0$ and $\min_{\overline\Omega}
\left(\overline\rho-\frac a2\right)>0$. There exist $\varepsilon,C>0$, depending only on
$\Omega,n,\beta,\gamma,\overline\rho,a,K,L$, such that the following holds. Suppose that $0\le\rho_0,\,\rho_1\in C^{0,\beta}(\overline\Omega)$ and $
\int_\Omega(\rho_1-\rho_0)\,dx=a|\Omega|$, and
$$
d:=
\|\rho_0-\overline\rho+a\|_{C^{0,\beta}(\overline\Omega)}
+
\|\rho_1-\overline\rho\|_{C^{0,\beta}(\overline\Omega)}
<\varepsilon.
$$
Then there exists
$M\in C^{1,\beta}(\overline\Omega;\mathbb R^n)$ such that $
M=\mathrm{Id}$ on $\partial\Omega$, $(M,a)\in
\mathcal V_K^{\rho_1}\cap\mathcal A_L^{\rho_1}$, and $$
\|M-\mathrm{Id}\|_{C^{1,\beta}(\overline\Omega)}
\le Cd.$$
\end{proposition}

\begin{proof}
\textit{Step 1.} Fix $0<\eta<\beta$.
For $s\in\{\eta,\beta\}$, set
$$
Z_s:=
\left\{
g\in C^{0,s}(\overline\Omega):
\int_\Omega g\,dx=0
\right\},
$$
and
$$
W_s:=
\left\{
u\in C^{1,s}(\overline\Omega;\mathbb R^n):
u=0\text{ on }\partial\Omega
\right\}.
$$
By the construction in
\cite[Theorem~2 and Remark~(ii)]{bibDM1990},
there exists a linear operator $\mathcal B$ such that,
for each $s\in\{\eta,\beta\}$ and $g\in Z_s$,
the vector field $\mathcal B g\in W_s$ satisfies
$$
\operatorname{div}(\mathcal B g)=g
\quad\text{in }\Omega,
\qquad
\mathcal B g=0
\quad\text{on }\partial\Omega,
$$
with the estimate $
\|\mathcal B g\|_{C^{1,s}}
\le C_s\|g\|_{C^{0,s}}$, where $C_s$ depends only on $\Omega$, $n$, and $s$,
and is independent of $g$. The same operator is used at both exponents. Write $\sigma:=\overline\rho-\frac a2$.
Since $\sigma$ is bounded away from zero, the operator $
\mathcal B_\sigma g:=\frac{\mathcal B g}{\sigma}$
satisfies
$$
\operatorname{div}(\sigma\mathcal B_\sigma g)=g\quad\text{and}\quad
\|\mathcal B_\sigma g\|_{C^{1,s}}
\le \widetilde{C_s}\|g\|_{C^{0,s}},
\qquad\text{for }s\in\{\eta,\beta\},
$$
where $\widetilde C_s$ depends only on
$\Omega,n,s$, $\|\sigma^{-1}\|_{C^{1,s}}$,
and the chosen H\"older norms.
In particular, it is independent of $g$.

\noindent
\textit{Step 2. }Choose $0<\theta<1$ such that
$$
1-\theta>L^{-1},
\qquad
1+\theta<L,
\qquad
n^{n/2}
\left(\frac{1+\theta}{1-\theta}\right)^n<K.
$$
If $u\in W_\beta$ and
$\|\nabla u\|_{L^\infty}\le\theta$, its zero extension
$\widetilde u$ to $\mathbb R^n$ is $\theta$-Lipschitz.
Thus, for every $t\in[0,1]$,
$$
(1-\theta)|x-y|
\le
|x+t\widetilde u(x)-y-t\widetilde u(y)|
\le
(1+\theta)|x-y|.
$$
For $x\in\mathbb R^n$, let $
F_t(x):=x+t\widetilde u(x)$. For every fixed $t\in[0,1]$ and $y\in\mathbb R^n$,
the map $
x\longmapsto y-t\widetilde u(x)$ is a contraction on $\mathbb R^n$, with Lipschitz constant at most $t\theta<1$.
It therefore has a unique fixed point $x$,
which satisfies $F_t(x)=y$.
Thus $F_t$ is bijective on $\mathbb R^n$.
Since $\widetilde u=0$ on
$\mathbb R^n\setminus\Omega$, every point outside
$\Omega$ is fixed by $F_t$.
If $x\in\Omega$ and $F_t(x)=y\notin\Omega$,
then $F_t(y)=y=F_t(x)$, contradicting injectivity.
Hence $F_t(\Omega)\subset\Omega$.
Conversely, the preimage of any $y\in\Omega$
cannot lie outside $\Omega$, where $F_t$ is the
identity.
Therefore $F_t(\Omega)=\Omega$.
Since $F_t$ also fixes $\partial\Omega$ pointwise, $
(\mathrm{Id}+tu)(\overline\Omega)=\overline\Omega$. Since $\|\nabla u\|_{L^\infty(\Omega)}\le\theta<1$,
we have $
|(I+t\nabla u(x))\xi|
\ge (1-t\theta)|\xi|
\ge (1-\theta)|\xi|$ for every $t\in[0,1]$ and $\xi\in\mathbb R^n$. Thus $I+t\nabla u(x)$ is invertible for every
$t\in[0,1]$. Thus, 
$$
\det(I+t\nabla u)>0
\quad\text{and}\quad\int_\Omega\det(I+t\nabla u)\,dx=|\Omega|.
$$

\noindent
\textit{Step 3. }For every $u\in W_\beta$ satisfying
$\|u\|_{C^{1,\beta}(\overline\Omega)}<r$,
where $r>0$ is sufficiently small, define
$$
\mathcal H(u):=
\overline\rho(\mathrm{Id}+u)\det(I+\nabla u)
-a\int_0^1\det(I+t\nabla u)\,dt
-(\overline\rho-a).
$$
The change-of-variables formula gives $\int_\Omega\mathcal H(u)\,dx=0$. Its linearization at zero is
$$
D\mathcal H(0)[u]
=
\nabla\overline\rho\cdot u
+\left(\overline\rho-\frac a2\right)
\operatorname{div}u
=
\operatorname{div}(\sigma u).
$$
Define the nonlinear remainder $
\mathcal N(u):=
\mathcal H(u)-\operatorname{div}(\sigma u)$. Both terms belong to $C^{0,\beta}(\overline\Omega)$ and have zero mean, so $\mathcal N(u)\in Z_\beta$. We next estimate its $C^{0,\beta}$ norm. Throughout the following estimates, all H\"older norms are taken over $\overline\Omega$. Set $
h_s(x):=
\nabla\overline\rho(x+su(x))
-\nabla\overline\rho(x)$. Since $\nabla\overline\rho\in C^{0,\gamma}$, $\|h_s\|_{C^0}
\le
[\nabla\overline\rho]_{C^{0,\gamma}}
s^\gamma\|u\|_{C^0}^{\gamma}$.
Moreover, the maps $\mathrm{Id}+su$ have uniformly bounded Lipschitz constants. Hence
$$
\begin{aligned}
|h_s(x)-h_s(y)|
&\le
[\nabla\overline\rho]_{C^{0,\gamma}}
\bigl(
|x+su(x)-y-su(y)|^\gamma
+|x-y|^\gamma
\bigr)\\
&\le C|x-y|^\gamma,
\end{aligned}
$$
and therefore $\|h_s\|_{C^{0,\gamma}}\le C$.
The constants are independent of $s$ and of $u$
in the chosen neighborhood.
Since $0<\beta<\gamma$, H\"older interpolation gives
$$
\|h_s\|_{C^{0,\beta}}\le
C\|h_s\|_{C^0}^{1-\beta/\gamma}
\|h_s\|_{C^{0,\gamma}}^{\beta/\gamma}\le
C s^{\gamma-\beta}
\|u\|_{C^0}^{\gamma-\beta}.
$$
Taylor's formula in integral form yields
$\overline\rho(\mathrm{Id}+u)
-\overline\rho
-\nabla\overline\rho\cdot u
=
\int_0^1 h_s\cdot u\,ds$. Using the product estimate in $C^{0,\beta}$,
we obtain
$$
\begin{aligned}
\left\|
\overline\rho(\mathrm{Id}+u)
-\overline\rho
-\nabla\overline\rho\cdot u
\right\|_{C^{0,\beta}}&\le
C\|u\|_{C^{0,\beta}}
\int_0^1\|h_s\|_{C^{0,\beta}}\,ds\\
&\le
C\|u\|_{C^{0,\beta}}
\|u\|_{C^0}^{\gamma-\beta}
\int_0^1s^{\gamma-\beta}\,ds\\
&\le
C\|u\|_{C^{1,\beta}}^{1+\gamma-\beta}.
\end{aligned}
$$
Write $\delta:=\gamma-\beta\in(0,1)$ and set $
q:=\|u\|_{C^{1,\beta}}$.
We work in a fixed neighborhood of zero on which $q\le r_0\le1$.
All constants below are independent of $u$
in this neighborhood. For $t\in[0,1]$, define $
P_t(u):=
\det(I+t\nabla u)-1-t\operatorname{div}u$. Every term in $P_t(u)$ has degree at least two
in $\nabla u$. The product estimate in
$C^{0,\beta}$ therefore gives
$$
\sup_{t\in[0,1]}
\|P_t(u)\|_{C^{0,\beta}}
\le
C\sum_{k=2}^n\|\nabla u\|_{C^{0,\beta}}^k
\le Cq^2.
$$
Set $
e(u):=
\overline\rho(\mathrm{Id}+u)
-\overline\rho-\nabla\overline\rho\cdot u$.
The preceding Taylor estimate yields $
\|e(u)\|_{C^{0,\beta}}\le Cq^{1+\delta}$.
Consequently, $
\|\overline\rho(\mathrm{Id}+u)-\overline\rho\|_{C^{0,\beta}}
\le Cq$ and $
\|\overline\rho(\mathrm{Id}+u)\|_{C^{0,\beta}}
\le C$. Using $
\det(I+\nabla u)
=1+\operatorname{div}u+P_1(u)$
and
$$
\int_0^1\det(I+t\nabla u)\,dt
=
1+\frac12\operatorname{div}u
+\int_0^1P_t(u)\,dt,
$$
we obtain
$$
\begin{aligned}
\mathcal N(u)
={}&e(u)
+\bigl(
\overline\rho(\mathrm{Id}+u)-\overline\rho
\bigr)\operatorname{div}u\\
&+\overline\rho(\mathrm{Id}+u)P_1(u)
-a\int_0^1P_t(u)\,dt.
\end{aligned}
$$
Hence
$$
\|\mathcal N(u)\|_{C^{0,\beta}}
\le C\bigl(q^{1+\delta}+q^2\bigr)
\le Cq^{1+\delta},
$$
because $q\le1$ and $1+\delta<2$.
Thus
$$
\|\mathcal N(u)\|_{C^{0,\beta}}
\le
C\|u\|_{C^{1,\beta}}^{1+\gamma-\beta}.
$$

\noindent
\textit{Step 4.} Now write
$
p:=\rho_1-\overline\rho$ and $h:=\rho_0-\rho_1+a$. The mass condition gives $h\in Z_\beta$, $
\|h\|_{C^{0,\beta}}\le d$ and $
\|p\|_{C^{0,\beta}}\le d$. Define
$$
\mathcal Q_p(u):=
p(\mathrm{Id}+u)\det(I+\nabla u)-p.
$$
Again, changing variables gives
$\mathcal Q_p(u)\in Z_\beta$.
On a fixed sufficiently small ball in $W_\beta$,
the composition and product estimates imply $
\|\mathcal Q_p(u)\|_{C^{0,\beta}}
\le C\|p\|_{C^{0,\beta}}$. This estimate requires no derivatives of $p$. Consider
\[
\mathcal T(u):=
\mathcal B_\sigma
\bigl(h-\mathcal N(u)-\mathcal Q_p(u)\bigr).
\]
The preceding estimates give a constant $C_0>0$,
independent of $\rho_0,\rho_1$, such that $
\|\mathcal T(u)\|_{C^{1,\beta}}
\le
C_0\left(
d+\|u\|_{C^{1,\beta}}^{1+\gamma-\beta}
\right)$ on that ball. Choose $r>0$ sufficiently small that $r\le\theta$,
the preceding estimates hold for
$\|u\|_{C^{1,\beta}}\le r$, and
$
C_0r^{\gamma-\beta}\le\frac12$.
Then choose $\varepsilon>0$ such that
$C_0\varepsilon\le r/2$.
It follows that $\mathcal T$ maps
$$
\mathcal K_r:=
\left\{
u\in W_\beta:
\|u\|_{C^{1,\beta}}\le r
\right\}$$
into itself. We apply the Schauder fixed point theorem in $W_\eta$.
The set $K_r$nis nonempty and convex. The compact embedding
$C^{1,\beta}(\overline\Omega)
\hookrightarrow C^{1,\eta}(\overline\Omega)$
makes it relatively compact in $W_\eta$. It is also closed in $W_\eta$:
under convergence in $C^{1,\eta}$, the zero
boundary values are preserved, and the
$C^{1,\beta}$ norm bound passes to the limit.
Hence $K_r$ is compact in $W_\eta$.
To verify continuity, let $u_j,u\in\mathcal K_r$
satisfy $u_j\to u$ in $C^{1,\eta}$.
Then $
\mathcal N(u_j)\to\mathcal N(u)$ in $C^{0,\eta}$. Also, $p(\mathrm{Id}+u_j)$ converges uniformly to
$p(\mathrm{Id}+u)$, and these compositions have
uniformly bounded $C^{0,\beta}$ norms.
Since $\eta<\beta$, interpolation gives $
p(\mathrm{Id}+u_j)\to p(\mathrm{Id}+u)$ in $C^{0,\eta}$. The determinants converge in $C^{0,\eta}$ as well.
Consequently, $
\mathcal Q_p(u_j)\to\mathcal Q_p(u)$ in $C^{0,\eta}$. The boundedness of $\mathcal B_\sigma$ at exponent
$\eta$ proves that $\mathcal T$ is continuous on
$\mathcal K_r$ in the $W_\eta$ topology. Schauder's theorem therefore gives
$u\in\mathcal K_r$ satisfying $\mathcal T(u)=u$.
At this fixed point,
$$
\begin{aligned}
\|u\|_{C^{1,\beta}}
&\le
C_0d+
C_0\|u\|_{C^{1,\beta}}^{1+\gamma-\beta}\\
&\le
C_0d+\frac12\|u\|_{C^{1,\beta}},
\end{aligned}
$$
and hence $
\|u\|_{C^{1,\beta}}\le2C_0d$. Set $M:=\mathrm{Id}+u$.
Applying $\operatorname{div}(\sigma\,\cdot)$
to the fixed point equation gives
$
\operatorname{div}(\sigma u)
=
h-\mathcal N(u)-\mathcal Q_p(u)$.
Thus $
\mathcal H(u)+\mathcal Q_p(u)=h$,
which is equivalent to
$$
\rho_1(M)\det\nabla M
=
\rho_0+a\int_0^1J_t[M]\,dt
=
R_1[M,a].
$$

\noindent
\textit{Step 5.} Since $\|\nabla u\|_{L^\infty}\le r\le\theta<1$
and $\Omega$ is convex, we have
\[
|u(x)-u(y)|\le\theta|x-y|
\qquad
\text{for all }x,y\in\overline\Omega.
\]
Consequently, for every $t\in[0,1]$,
\[
\begin{aligned}
|X_t[M](x)-X_t[M](y)|
&\ge (1-t\theta)|x-y|
\ge (1-\theta)|x-y|,\\
|X_t[M](x)-X_t[M](y)|
&\le (1+t\theta)|x-y|
\le (1+\theta)|x-y|.
\end{aligned}
\]
By the choice of $\theta$, these estimates imply
the prescribed bi-Lipschitz bounds with constant $L$.
Moreover, $u=0$ on $\partial\Omega$, so the
zero-extension argument above gives
\[
X_t[M](\overline\Omega)=\overline\Omega
\qquad\text{for every }t\in[0,1].
\]
Thus the required flow geometry holds. It remains to verify the nonnegativity of the
intermediate states.
For each fixed $x\in\Omega$,
\[
\partial_tR_t[M,a](x)=aJ_t[M](x).
\]
Since $J_t[M](x)>0$, the function
$t\mapsto R_t[M,a](x)$ is nondecreasing if $a\ge0$
and nonincreasing if $a\le0$.
Its endpoint values satisfy
\[
R_0[M,a]=\rho_0\ge0,
\qquad
R_1[M,a]=\rho_1(M)\det\nabla M\ge0.
\]
Therefore
\[
R_t[M,a](x)
\ge
\min\{\rho_0(x),\rho_1(M(x))\det\nabla M(x)\}
\ge0
\]
for every $t\in[0,1]$.
Hence $(M,a)\in\mathcal A_L^{\rho_1}$. Finally,
\[
\frac{|\nabla M|^n}{\det\nabla M}
\le
n^{n/2}
\left(\frac{1+\theta}{1-\theta}\right)^n<K,
\qquad
\int_\Omega\det\nabla M\,dx=|\Omega|.
\]
Hence $(M,a)\in\mathcal V_K^{\rho_1}$.
The asserted estimate follows from
$\|u\|_{C^{1,\beta}}\le2C_0d$.
\end{proof}

\begin{remark}
The regularity assumption
$\overline\rho\in C^{1,\gamma}(\overline\Omega)$
concerns only the fixed reference density.
Both prescribed endpoint densities may be merely
$C^{0,\beta}$.
The proposition provides a local sufficient condition
and does not assert feasibility for arbitrary H\"older densities.
\end{remark}

\begin{proof}[Proof of Theorem \ref{cor:local-existence-convergence}]
Choose $\varepsilon$ as in
Proposition~\ref{prop:local-feasibility}.
If $d<\varepsilon$, that proposition provides
a finite-energy pair $
(M_*,a)\in
\mathcal V_K^{\rho_1}\cap\mathcal A_L^{\rho_1}$. Existence and static-dynamic equivalence follow
from Theorems~\ref{GTH5.2} and~\ref{TH1.4}.
The mass constraint gives
\[
\|f\|_{L^2(0,1)}^2
=
a^2+\|f-a\|_{L^2(0,1)}^2
\]
for every $(M,f)\in\mathcal A_L^{\rho_1}$.
The constant-source competitor therefore implies
$$
b_L=a^2,
\qquad
\mathcal S_L^{\rho_1}
=
\left\{(M,a)\in\mathcal A_L^{\rho_1}\right\},
\qquad
g_L=m_a.
$$
The convergence and minimum-value expansion
now follow from
Theorem~\ref{thm:minimum-source-limit}.
\end{proof}

\subsection{Examples of nonemptiness}
We give explicit examples verifying the feasibility
assumptions used in the existence and convergence results.
The first two constructions establish nonemptiness
of $\mathcal V_K^{\rho_1}$ and
$\mathcal V_{K,\Lambda}^{\rho_1}$.
For the reconstruction class, we exhibit endpoint
densities admitting both constant and nonconstant sources,
followed by an example for which admissible sources
exist but no constant source is feasible.
These examples distinguish the constant-source limit
from the general minimum-source-cost limit.

\begin{example}[Nonemptiness of $\mathcal V_K^{\rho_1}$]
Let $\Omega=(0,1)^n$ and suppose that
$$
\rho_0(x)=r(x_1),
\qquad
\rho_1(y)=q(y_1),
$$
where $r\in L^\infty(0,1)$ is nonnegative and
$q\in C([0,1])$ is strictly positive.
Write
$$
r_*:=\operatorname*{ess\,inf}_{(0,1)}r,
\qquad
r^*:=\|r\|_{L^\infty(0,1)},
$$
and
$$
q_*:=\min_{[0,1]}q,
\qquad
q^*:=\max_{[0,1]}q.
$$
Set
$$
a:=\int_0^1q(s)\,ds-\int_0^1r(s)\,ds
$$
and assume that
\begin{equation}\label{e6.1}
r_*+\frac a2>0,
\qquad
q_*-\frac a2>0.
\end{equation}
Define
$$
Q_a(y):=\int_0^y\left(q(\xi)-\frac a2\right)d\xi
$$
and
$$
F(s):=
Q_a^{-1}\left(
\int_0^s\left(r(\xi)+\frac a2\right)d\xi
\right).
$$
The assumptions \eqref{e6.1} ensure that both cumulative
functions are strictly increasing.
Furthermore,
$$
\int_0^1\left(r+\frac a2\right)ds
=
\int_0^1\left(q-\frac a2\right)ds.
$$
Thus $F$ is a Lipschitz strictly increasing map
of $[0,1]$ onto itself. For almost every $s\in(0,1)$, its derivative is
$$
g(s):=F'(s)
=
\frac{r(s)+a/2}{q(F(s))-a/2}.
$$
In particular,
$$
0<\ell:=
\frac{r_*+a/2}{q^*-a/2}
\le g(s)\le
l:=
\frac{r^*+a/2}{q_*-a/2}
<\infty.
$$
Set
$$
M(x):=(F(x_1),x_2,\ldots,x_n),
\qquad
f(t):=a.
$$
Then $M$ maps $\Omega$ onto itself and belongs to
$W^{1,\infty}(\Omega;\mathbb R^n)$.
Moreover,
$$
\nabla M
=
\operatorname{diag}(g(x_1),1,\ldots,1),
\qquad
\det\nabla M=g(x_1),
$$
and
$$
\det\bigl(t\nabla M+(1-t)I\bigr)=1-t+t g(x_1)
\ge\min\{1,\ell\}>0.
$$
The definition of $g$ gives
$$
q(F(x_1))g(x_1)
=
r(x_1)+\frac a2\bigl(1+g(x_1)\bigr).
$$
Consequently,
$$
\begin{aligned}
R_1[M,f](x)
&=
r(x_1)+\int_0^1
a\bigl(1-s+s g(x_1)\bigr)\,ds\\
&=
q(F(x_1))g(x_1)\\
&=
\rho_1(M(x))\det\nabla M(x).
\end{aligned}
$$
For intermediate times,
$$
R_t[M,f](x)
=
r(x_1)+
a\left(t+\frac{t^2}{2}(g(x_1)-1)\right).
$$
Since $\det\bigl(t\nabla M+(1-t)I\bigr)>0$ and $a$ is constant,
this function is monotone in $t$.
Its endpoint values are $r(x_1)$ and
$q(F(x_1))g(x_1)$, both nonnegative.
Hence
$$
R_t[M,f](x)\ge0
\quad\text{for every }t\in[0,1].
$$
The Jacobian mass bound follows from
$$
\int_\Omega\det\nabla M\,dx
=
\int_0^1g(s)\,ds
=
F(1)-F(0)=1=|\Omega|.
$$
Also,
$$
|\nabla M|^n
=
\bigl(g(x_1)^2+n-1\bigr)^{n/2}
\le K_*\det\nabla M,
$$
where
$$
K_*:=\frac{(l^2+n-1)^{n/2}}{\ell}.
$$
Therefore,
$$
(M,f)\in\mathcal V_K^{\rho_1}
\qquad\text{for every }K\ge K_*.
$$
Its energy is finite because the map has bounded range.
\end{example}

\begin{example}[Nonemptiness of $\mathcal V_{K,\Lambda}^{\rho_1}$]
\label{ex:whole-space-affine}
Let $\Omega\subset\mathbb R^n$ be a bounded connected
open set with Lipschitz boundary.
Fix $K\ge n^{n/2}$ and $\Lambda>0$, and choose $s>0$
such that
$$
s^n|\Omega|\le\Lambda.
$$
For any $b\in\mathbb R^n$, set
$$
M(x):=b+sx,
\qquad
j_s(t):=(1-t+ts)^n,
\qquad
c_s:=\int_0^1j_s(t)\,dt>0.
$$
Choose $c>0$ and $q\ge0$, and take
$\psi\in C_c^\infty(\mathbb R^n)$ satisfying
$$
0\le\psi\le1,
\qquad
\psi=1
\quad\text{on }b+s\overline\Omega.
$$
Define
$$
\rho_0(x):=c,
\qquad
\rho_1(y):=\frac{c+qc_s}{s^n}\psi(y),
\qquad
f(t):=q.
$$
Then $\rho_1\in C_c^\infty(\mathbb R^n)$ is
nonnegative, and
$$
\int_\Omega\rho_0\,dx=c|\Omega|>0.
$$
Since $\nabla M=sI$, we have
$$
\det\nabla M=s^n,
\qquad
J_t[M]=j_s(t)>0,
$$
and
$$
R_t[M,q]
=c+q\int_0^t j_s(r)\,dr
\ge c>0.
$$
As $\psi(M(x))=1$, the endpoint constraint holds
pointwise:
$$
\rho_1(M(x))\det\nabla M(x)
=c+qc_s
=R_1[M,q](x).
$$
Moreover,
$$
|\nabla M|^n
=n^{n/2}s^n
\le K\det\nabla M,
$$
and
$$
\int_\Omega\det\nabla M\,dx
=s^n|\Omega|\le\Lambda.
$$
Therefore,
$$
(M,q)\in\mathcal V_{K,\Lambda}^{\rho_1}.
$$
This pair has finite energy, since
$$
E_\alpha(M,q)=
\left(c+q\int_0^1(1-t)j_s(t)\,dt\right)
\int_\Omega|b+(s-1)x|^2\,dx
+\frac{q^2}{\alpha}<\infty.
$$
Taking $q>0$ gives a positive constant-source pair,
whereas $q=0$ gives a finite-cost zero-source pair.
The translation $b$ is unrestricted, so $M(\Omega)$
need not lie in $\overline\Omega$.
\end{example}

\begin{example}[Constant and nonconstant sources for the same
endpoint densities]
\label{ex:constant-and-nonconstant-sources}
Let $\Omega=(0,1)^n$, fix $L>1$ and $0<a<r$, and choose
$$
0<\varepsilon<
\min\left\{
1-L^{-1},
\frac{r-a}{r-a/2}
\right\}.
$$
Set
$$
\rho_1\equiv r,
\qquad
\rho_0(x)
=
r-a+\varepsilon\left(r-\frac a2\right)
\cos(2\pi x_1),
$$
and define
$$
M_*(x)=
\left(
x_1+\frac{\varepsilon}{2\pi}\sin(2\pi x_1),
x_2,\ldots,x_n
\right).
$$
Then $\rho_0$ is strictly positive and
$$
\int_\Omega(\rho_1-\rho_0)\,dx=a.
$$
Each interpolation map changes only the first coordinate,
whose derivative is $
1+t\varepsilon\cos(2\pi x_1)$.
This coordinate map fixes $0$ and $1$ and is strictly
increasing. Hence $X_t[M_*]$ maps $\overline\Omega$
onto itself. Moreover,
$$
L^{-1}<1-\varepsilon
\le 1+t\varepsilon\cos(2\pi x_1)
\le 1+\varepsilon<L,
$$
so the prescribed bi-Lipschitz bounds hold.
The interpolation Jacobian is
$$
J_t[M_*](x)=1+t\varepsilon\cos(2\pi x_1).
$$

For the constant source $a$, we have
$$
R_t[M_*,a](x)
=
\rho_0(x)
+
a\left(
t+\frac{\varepsilon t^2}{2}\cos(2\pi x_1)
\right).
$$
Since $a>0$ and $J_t[M_*]>0$,
$$
R_t[M_*,a]\ge\rho_0>0.
$$
At $t=1$,
$$
R_1[M_*,a]
=
r\bigl(1+\varepsilon\cos(2\pi x_1)\bigr)
=
\rho_1(M_*)\det\nabla M_*.
$$
Thus $(M_*,a)\in\mathcal A_L^{\rho_1}$.

Now consider the nonconstant source $f(t):=6at(1-t)$.
It is nonnegative and satisfies
$$
\int_0^1f(t)\,dt=a,
\qquad
\int_0^1tf(t)\,dt=\frac a2.
$$
Consequently,
$$
R_t[M_*,f]
=
\rho_0+\int_0^t
f(s)\bigl(1+s\varepsilon\cos(2\pi x_1)\bigr)\,ds
\ge\rho_0>0,
$$
and the two moment identities give
$$
\begin{aligned}
R_1[M_*,f]
&=\rho_0+a+\frac{a\varepsilon}{2}\cos(2\pi x_1)\\
&=r\bigl(1+\varepsilon\cos(2\pi x_1)\bigr)\\
&=\rho_1(M_*)\det\nabla M_*.
\end{aligned}
$$
Therefore $(M_*,f)\in\mathcal A_L^{\rho_1}$ as well.
Both pairs have finite energy. The constant-source problem also has $m_a>0$.
Indeed, let $c_0:=\min_{\overline\Omega}\rho_0>0$.
For every $(M,a)\in\mathcal A_L^{\rho_1}$,
the positivity of $a$ and $J_t[M]$ gives
$$
G(M,a)\ge c_0\int_\Omega|M(x)-x|^2\,dx.
$$
If $m_a=0$, an attained minimizer would satisfy
$M=\mathrm{Id}$ a.e.
Its endpoint constraint would then imply
$\rho_1=\rho_0+a$, contrary to the definition of $\rho_0$.
Thus a constant source is feasible, but nonzero
spatial transport is still necessary.
\end{example}

\begin{example}[A nonempty reconstruction class without
a constant source]
\label{ex:no-constant-source}
Let $\Omega=(0,1)^n$, $L=2$, and set
$$
\rho_1\equiv1,
\qquad
\rho_0(x)=2+2\cos(2\pi x_1).
$$
Then $\rho_0\ge0$ and the prescribed mean source is
$$
a=
\frac{\int_\Omega\rho_1\,dx-\int_\Omega\rho_0\,dx}
{|\Omega|}
=-1.
$$
Define
$$
M_*(x)=
\left(
x_1+\frac{1}{4\pi}\sin(2\pi x_1),
x_2,\ldots,x_n
\right),
\qquad
f_*(t)=14-30t.
$$
The first coordinate of $X_t[M_*]$ fixes $0$ and $1$
and has derivative
$$
1+\frac t2\cos(2\pi x_1)
\in\left[\frac12,\frac32\right].
$$
The remaining coordinates are unchanged.
Thus $X_t[M_*]$ maps $\overline\Omega$ onto itself
and satisfies the prescribed bi-Lipschitz bounds
with $L=2$.
Moreover,
$$
J_t[M_*](x)=1+\frac t2\cos(2\pi x_1).
$$
The source satisfies
$$
\int_0^1f_*(t)\,dt=-1,
\qquad
\int_0^1tf_*(t)\,dt=-3.
$$
Hence
$$
\begin{aligned}
R_1[M_*,f_*]
&=\rho_0-1-\frac32\cos(2\pi x_1)\\
&=1+\frac12\cos(2\pi x_1)\\
&=\rho_1(M_*)\det\nabla M_*.
\end{aligned}
$$
Since $J_t[M_*]>0$ and $f_*$ changes sign once,
from positive to negative,
$R_t[M_*,f_*](x)$ first increases and then decreases
as a function of $t$. Therefore,
$$
R_t[M_*,f_*](x)
\ge
\min\{\rho_0(x),R_1[M_*,f_*](x)\}
\ge0.
$$
It follows that
$$
(M_*,f_*)\in\mathcal A_2^{\rho_1},
$$
and this pair has finite energy. On the other hand, if $(M,-1)\in\mathcal A_2^{\rho_1}$,
the endpoint constraint and the Jacobian lower bound
would give
$$
\rho_0
=
J_1[M]+\int_0^1J_t[M]\,dt
\ge 2^{1-n}
\quad\text{a.e. in }\Omega.
$$
This contradicts
$$
\operatorname*{ess\,inf}_{\Omega}\rho_0=0.
$$
Thus $\mathcal A_2^{\rho_1}$ is nonempty but contains
no constant-source admissible pair. Finally, Proposition~\ref{P5.2}, with $c_0=0$,
$c_1=1$, and $a=-1$, yields $
b_2\ge1+2^{2-4n}$. The constructed competitor gives
$$
b_2\le\|f_*\|_{L^2(0,1)}^2
=\int_0^1(14-30t)^2\,dt=76.
$$
Consequently,
$$
1+2^{2-4n}\le b_2\le76.
$$
This example shows that admissibility may require
temporal variation of the source even when its mean
is fixed by the endpoint masses.
\end{example}

\vspace{10mm}

\noindent {\bf Acknowledgements.}
The first author is grateful to Prof. Xiao-Ping Yang for insightful discussions. Both authors also wish to express their sincere gratitude to him for his long-term encouragement and guidance in this research direction.
This work has been supported by the National Natural Science Foundation of China (Grant No. 12271093), the Jiangsu Provincial Scientific Research Center of Applied Mathematics (Grant No. BK20233002), and Shanghai Institute for Mathematics and Interdisciplinary Sciences (SIMIS) under grant number SIMIS-ID-2025-AD.

\bibliographystyle{amsplain}

\end{document}